\documentclass[a4paper,10pt]{amsart}
\usepackage[hmargin=1 in, vmargin = 1 in]{geometry}
\usepackage[utf8]{inputenc}
\usepackage{xcolor}
\usepackage{latexsym}
\usepackage{lipsum}
\usepackage{nicefrac}
\usepackage[english]{babel}
\usepackage[T1]{fontenc}
\usepackage{amsmath}
\usepackage{amssymb}
\usepackage{fge}
\usepackage{mathtools}
\usepackage{amsthm}
\usepackage{verbatim}
\usepackage{mathrsfs}
\usepackage[cal=boondox,scr=boondoxo]{mathalfa}
\usepackage{braket}
\usepackage{tikz-cd}
\usepackage{tikz}
\usepackage{float}
\usepackage{tkz-euclide}
\usetikzlibrary{ shapes }
\usepackage[bookmarksopen=true]{hyperref}
\definecolor{darkgreen}{RGB}{0,80,0}
\hypersetup{
colorlinks,%
citecolor=darkgray,%
filecolor=darkgray,%
linkcolor=darkgray,%
urlcolor=darkgray,%
}
\usepackage[noabbrev]{cleveref}
\usepackage[font={small}]{caption}
\usepackage{verbatim}
\usepackage{mathrsfs}
\usepackage[cal=boondox,scr=boondoxo]{mathalfa}
\usepackage{enumitem}
\setlist[enumerate]{itemsep=5pt, labelindent=0pt, leftmargin=*, topsep=5pt, itemindent=\parindent}
\usepackage[utf8]{inputenc}
\usepackage[numbers,sort]{natbib}

\usepackage{nicematrix}
\usetikzlibrary{patterns}
\usetikzlibrary{positioning}

\newtheorem{numerable}{Numerable}
\newtheorem{lemma}[numerable]{Lemma}
\newtheorem{proposition}[numerable]{Proposition}
\newtheorem{theorem}[numerable]{Theorem}
\newtheorem{corollary}[numerable]{Corollary}
\newtheorem{conjecture}{Conjecture}

\numberwithin{numerable}{section}
\theoremstyle{definition}
\newtheorem{definition}[numerable]{Definition}

\newtheorem{convention}[numerable]{Convention}
\newtheorem{remark}[numerable]{Remark}

\numberwithin{equation}{section}
\definecolor{details}{RGB}{0,0,255}
\definecolor{task}{RGB}{0,191,0}
\definecolor{sketch}{RGB}{255,0,0}

\newcommand{\CB}{\ensuremath{\mathcal{B}}}
\newcommand{\CC}{\ensuremath{\mathcal{C}}}
\newcommand{\D}{\ensuremath{\Delta}}
\newcommand{\Der}{\operatorname{Der}}
\newcommand{\Hom}{\operatorname{Hom}}
\newcommand{\K}{\ensuremath{\mathbb{K}}}
\newcommand{\Kosz}{\operatorname{Kosz}}

\newcommand{\N}{\ensuremath{\mathbb{N}}}%{\mathbb{Z}_{\geqslant{0}}}
\newcommand{\M}{\ensuremath{\mathcal{M}}}
\newcommand{\R}{\mathbb{R}}
\newcommand{\Syz}{\operatorname{Syz}}
\newcommand{\CU}{\ensuremath{\mathcal{U}}}
\newcommand{\Z}{\mathbb{Z}}
\newcommand{\bfa}{\ensuremath{\mathbf{a}}}
\newcommand{\bfb}{\ensuremath{\mathbf{b}}}
\newcommand{\bfc}{\ensuremath{\mathbf{c}}}

\newcommand{\card}[1]{\#{#1}}

\newcommand{\cone}{\ensuremath{\operatorname{cone}}}
\newcommand{\deff}[1]{{\textbf{#1}}}%highlight defined notion
\newcommand{\defi}[1]{\textbf{#1}}
\newcommand{\link}{\ensuremath{\operatorname{link}}}
\newcommand{\lk}{\ensuremath{\operatorname{link}}}

\newcommand{\opensimplex}[1]{\ensuremath{\braket{#1}}}

\newcommand{\rank}{\ensuremath{\operatorname{rank}}}
\newcommand{\sdif}{\ensuremath{\,\scalebox{.8}{$\fgebackslash$}\,}}

\newcommand{\sodass}{~\,:\,~}
\newcommand{\str}{\ensuremath{\operatorname{star}}}

\newcommand{\supp}{\operatorname{supp}}
\newcommand{\then}{\ensuremath{\Rightarrow}}
\newcommand{\back}{\ensuremath{\Leftarrow}}
\newcommand{\too}{\ensuremath{\longrightarrow}}
\newcommand{\wt}[1]{\widetilde{#1}}

\title{Simplicial complexes and matroids with vanishing $T^2$}
\author[A. Constantinescu]{Alexandru Constantinescu}
\address[AC]{Mathematics Institute, Freie Universit\"at Berlin, Berlin, Germany}
\email{alexandru.constantinescu@fu-berlin.de }
        
\author[P. Klein]{Patricia Klein}
\address[PK]{Department of Mathematics, Texas A\&M University, College Station, TX, USA}
\email{pjklein@tamu.edu}

\author[T.T.Nguy$\tilde{\text{\^E}}$n]{Th\'ai Th\`anh Nguy$\tilde{\text{\^E}}$n}
\address[TTN]{Department of Mathematics, University of Dayton, Dayton, OH, USA and \\
	University of Education, Hue University, 34 Le Loi St., Hue, Viet Nam}
\email{tnguyen5@udayton.edu}

\author[A.Singh]{Anurag Singh}
\address[AS]{Department of Mathematics, Indian Institute of Technology Bhilai, Durg, Chhattisgarh, India}
\email{anurags@iitbhilai.ac.in}

\author[L. Venturello]{Lorenzo Venturello}
\address[LV]{Department of Information Engineering and Mathematical Sciences, Universit\`a di Siena, Italy}
\email{lorenzo.venturello@unisi.it}
        
\date{\today}
\thanks{Klein was partially supported by NSF grant DMS-2246962. Nguy$\tilde{\text{\^e}}$n was partially supported by the NAFOSTED (Vietnam) grant 101.04-2023.07. Part of this work was done when Nguy$\tilde{\text{\^e}}$n was visiting the Vietnam Institute of Advanced Study in Mathematics (VIASM). He thanks VIASM for its hospitality. Singh was partially supported by the NBHM travel grant 0207/2/2023/R$\&$D-II/1854. Venturello was partially supported by INdAM-GNSAGA and by PRIN 2022S8SSW2}

\begin{document}
\begin{abstract}
  We investigate quotients by radical monomial ideals for which $T^2$, the second cotangent cohomology module, vanishes.
  The dimension of the graded components of $T^2$, and thus their vanishing,  depends only on the combinatorics of the corresponding simplicial complex.
  We give both a complete characterization and a full list of one dimensional complexes with $T^2=0$.
  We characterize the graded components of $T^2$ when the simplicial complex is a uniform matroid.
  Finally, we show that $T^2$ vanishes for all matroids of corank at most two and conjecture that all connected matroids with vanishing $T^2$ are of corank at most two.
\end{abstract}
\maketitle

\section{Introduction}
Let \K\ be an arbitrary field, and let $\K[\D]$ denote the Stanley-Reisner ring of a simplicial complex $\D$ on $[n]:=\{1,2,\ldots ,n \}$.
The cotangent cohomology modules of $\K[\D]$ have the structure of a \K-vector space.
The $\Z^n$-grading on $\K[\D]$ transfers to these modules, and the multigraded components are finite-dimensional vector spaces.
If one is interested in using cotangent cohomology to understand the deformations of Stanley-Reisner schemes, then some caution is needed when choosing the field \K.
In this paper, however, we are only interested in the vector space dimensions of the multigraded components of $T^2$.
These depend only on the combinatorics of $\D$ \cite{altmann2000stanleyreisner} because basis elements are in bijection with monomials related to $\D$.
For this reason, we will write simply $T^i(\D)$, without reference to the field, for the cotangent cohomology of $\K[\D]$. 

All obstructions for lifting first order deformations of $\textup{Spec}(\K[\D])$ are contained in $T^2$.
For embedded deformations of the associated projective scheme, these are located in $T^2_0$, the (total) degree zero component.
While $T^2$ may be in general larger than the obstruction space,
its vanishing has positive consequences for the deformations of these combinatorially defined schemes.
In particular, if $T^2$ vanishes, then the corresponding point on the Hilbert scheme is smooth.
For this reason, simplicial complexes with vanishing $T^2$ are said to be \emph{unobstructed}.
Note, however, that being unobstructed is not equivalent to  $T^2=0$, but is only a consequence thereof.

Monomial ideals with vanishing $T^2$ have been studied by other authors as well.
For instance, 
Christophersen and Ilten used degenerations of Mukai varieties to find unobstructed Fano Stanley-Reisner schemes and proved that the boundary complex of the dual polytope of the associahedron has trivial $T^2$ \cite{CI14}.
Ilten, N\'ajera Ch\'avez, and Treffinger proved the vanishing of $T^2$ for Stanley-Reisner rings associated to cluster complexes \cite{INT21}.
Nematbakhsh  identified special classes of quadratic monomial ideals with vanishing $T^2$ \cite{Nem16}, and
Fl{\o}ystad and Nematbakhsh studied letterplace ideals which are unobstructed \cite{FN18}.
Our approach here is to systematically look at low dimensional simplicial complexes with vanishing $T^2$, and to characterize matroids with vanishing $T^2$.

We start by introducing the minimally necessary terminology and background in Section~\ref{sec:prelim}.
Most notably, we recall the results of Altmann and Christophersen from~\cite{altmann2000stanleyreisner}, which we will use.
The two tools from that paper that will be most useful to us are the isomorphism  between the multigraded components of $T^i(\D)$ and those in purely non-positive multidegrees%
\footnote{~This means that all the entries of the multidegree $\bfb\in\Z^n$ are less than or equal to zero.}
of links of faces of $\D$  (see (\ref{eq:TiOfLink})) and the isomorphism between such components and the relative cohomology  of some combinatorially defined topological spaces (see Theorem~\ref{thm:TiAsRelativeCohomology}).

The main goal in Section~\ref{sec:lowDimensions} is to understand which one-dimensional complexes have a vanishing $T^2$.
For zero-dimensional complexes, $T^2=0$ is equivalent to having at most 3 vertices (see Lemma~\ref{lem:dim0T20}).
This gives us the first of three conditions which characterizes the vanishing of $T^2$ for one-dimensional complexes:
the local degree of each vertex is at most three.
Two more necessary conditions follow from the interpretation of relative cohomology (cf. Theorem~\ref{thm:TiAsRelativeCohomology}).
In Theorem~\ref{thm:unobstructedDim1} we show that the three conditions we found are also sufficient for the vanishing of $T^2$, thus obtaining a full theoretical characterization.
In Subsection~\ref{subsec:fullListDim1}, we use this characterization to work out a complete list of the one-dimensional simplicial complexes with vanishing $T^2$ (see Figure~\ref{fig:unobstructed1dim}).
 
In the last sections of this paper, we study $T^2(\D)$ when \D\ is a matroid, that is, when the faces of \D\ satisfy the independence axioms of matroids. In Proposition~\ref{prop:uniformT2}, we fully determine the purely non-positive multigraded components of $T^2$ for uniform matroids. Given that all links in a uniform matroid are also uniform, this gives a full characterization of $T^2$ for uniform matroids.
We will see that non-vanishing non-positively graded components can only be found in total degree $-2$,
and that uniform matroids have a vanishing $T^2$ if and only if they are of corank at most two.
In the final section, we prove that one direction of this equivalence holds for arbitrary matroids:
If a matroid $\M$ has corank at most two, then $T^2(\M)=0$ (see Proposition~\ref{prop:Corank2IsUnobstructed}).

If a simplicial complex is the join of two other complexes, then the cotangent cohomology modules that appear are closely related (see (\ref{eq:TofJoins})). In particular, $T^2$ of the join vanishes if and only if $T^2$ vanishes for the two joined complexes.
Matroids that are not the join (as simplicial complexes) of two other matroids are called connected.
We conjecture that every matroid with $T^2=0$ is the join of connected matroids of corank at most two.
In other words, if $T^2(\M)=0$ and $\M$ is connected, then $\textup{corank}(\M)\leqslant 2$ (Conjecture~\ref{conj:allUnobstructedMatroids}).

\subsection*{Acknowledgements}
We would like to thank the Banff International Research Station and in particular the organizers of the
\emph{Interactions Between Topological Combinatorics and Combinatorial Commutative Algebra} Workshop in April 2023:
Mina Bigdeli, 
Sara Faridi,
Satoshi Murai,
and Adam Van Tuyl,
for creating such an inspiring  working environment.  We are also very grateful to Ayah Almousa, who was part of our working group during the BIRS workshop. Her expertise and insights were invaluable.

We would also like to thank Jan Christophersen and Nathan Ilten for helpful communications, especially regarding \cite{CI14}.

\section{Preliminaries}
\label{sec:prelim}
\subsection{Combinatorics}
Let $E$ be a finite set.
An \deff{abstract simplicial complex}\footnote{~We will  usually drop the word \emph{abstract} and just use \deff{simplicial complex}.} on the vertex set $E$ is a subset $\Delta\subseteq 2^E$ of the power set of $E$, satisfying the condition:
\begin{enumerate}[label={({\bf I\arabic*})}]
\item\label{item:indep1}\textit{If $I\in \D$ and $J\subseteq I$, then $J\in\D$.}
\end{enumerate}
Unless otherwise stated,
we  assume that  $E=[n]=\set{1,\dots,n}$  for some positive integer $n$.
The subsets of $[n]$ which are elements of $\Delta$ are called \deff{faces} of $\Delta$.
The \deff{facets} are the faces which are maximal under inclusion.
A subset $C\subseteq [n]$ is a {nonface} of \D\ if $C\notin \D$;
if all proper subsets of $C$ are in \D, then $C$ is called a \deff{minimal nonface}.
A \deff{loop} is an element $v\in[n]$ with $\set{v}\notin \D$; equivalently, $v$ is not contained in any face of \D. A \deff{coloop} is a vertex $v\in [n]$ which is contained in every facet; alternatively, a coloop is not contained in any minimal nonface.
For every subset $W\subseteq[n]$, the \deff{restriction} of \D\ to  $W$ is the simplicial complex on $W$ given by
\[
  \D|_W := \Set{F \in \D\sodass F \subseteq W}.
\]
The \deff{deletion} of $W$ is the restriction to the complement of $W$ in $[n]$:
\[
  \D\sdif W :=\D|_{[n]\sdif W}.
\]

Given two simplicial complexes \D\ and $\Delta'$ on disjoint sets $E$ and $E'$, respectively, their \defi{join} is the simplicial on $E\sqcup E'$ given by
\[
  \Delta\ast \Delta' := \Set{F\sqcup F'\sodass F\in \Delta \text{~and~}F'\in\Delta'},
\]
where $\sqcup$ stands for the disjoint union. The \defi{link}\footnote{~This is a particular case of contraction, which can be defined for every subset $[n]$, not just for faces. For our purposes here, the link of a face will suffice.} of a face $F\in \Delta$ is defined as 
\[
    \link_\D{F} :=  \Set{A \in \D\sodass A \cap F = \emptyset\text{~and~} A \cup F \in \D}.
\]
For every finite set $F$, the abstract simplex on $F$ is $2^F=\Set{A\subseteq F}$. The \deff{star} of a face $F\in\Delta$ is
\[
  \str_\Delta F :=  2^F \ast \lk_\Delta F  = \Set{G \in\Delta \sodass F\cup G\in\Delta}.
\]

A \deff{matroid} is a nonempty\footnote{~This means that $\Delta\neq\emptyset$. If a simplicial complex satisfies $\Delta\neq \emptyset$, then by \ref{item:indep1} we have $\emptyset\in\Delta$. For matroids this last condition is usually included as an axiom.}
 simplicial complex \M\ whose faces satisfy the extra axiom: 
 \begin{enumerate}[label={({\bf I\arabic*})}, start=2]
 \item\label{item:indep2}\textit{If $I, J \in\M$  and $\card{J}<\card{I}$, then there exists $v\in I\sdif J$ such that $J\cup\set{v}\in\M$.}
 \end{enumerate} 
We will use \D\ to denote simplicial complexes that are not necessarily matroids and reserve the notation \M\ for matroids.
 We will call the faces of a matroid  \deff{independent sets} and the facets of a matroid \deff{bases}. The minimal nonfaces of a matroid are called \deff{circuits}.
In accordance with the matroid terminology, we will use the following notation for all simplicial complexes:
\begin{eqnarray*}
  \CC_{\D}&:=&\Set{C\subseteq [n]\sodass C\text{~is a minimal nonface of~}\D},\\
  \CB_\D&:=&\Set{B\subseteq [n]\sodass B\text{~is a facet  of~}\D}.
\end{eqnarray*}
 Matroids have equivalent characterizations in terms  of their circuits (cf. \cite[Section\,1.1]{Oxl11}) or of their bases (cf. \cite[Section\,1.2]{Oxl11}).

 \subsection{Algebra}
Recall that \K\ denotes a fixed arbitrary field throughout this document. Fix $n\in\Z_{>0}$. We denote by $S=\K[x_1,\dots,x_n]$ the polynomial ring in $n$ variables with coefficients in \K. To every simplicial complex \D\ on $[n]$, we associate a radical monomial ideal of $S$ called its \deff{Stanley-Reisner ideal}:
\[
  \textstyle
  I_\D :=\left\langle~  \prod_{i \in F} x_i \sodass F \in 2^{[n]}\sdif \D \right\rangle \subseteq S.
\]
This association is a one-to-one correspondence between  simplicial complexes on $[n]$ and radical monomial ideals of $S$. The quotient ring $\K[\D]=S/I_\D$ is called the \deff{Stanley-Reisner ring} of \D\ over the field \K.\\[1ex]
\indent We will now introduce the first and the second cotangent cohomology modules for Stanley-Reisner rings in the ad hoc way of \cite{altmann2000stanleyreisner}. For the general homological theory we refer to the books of Andr\'e \cite{And74} and of Loday \cite{Lod13}, and for the connection to deformation theory we refer to Hartshorne's book \cite{HAr09} and Sernesi's book \cite{Ser07}.
While some  algebraic structures related to Stanley-Reisner rings depend on the choice of field, the \K-vector space  dimensions of the cotangent cohomology modules depend only on the combinatorics of the complex \cite[Corollary\,1.4]{altmann2016rigidity}.
Since our focus is solely on these dimensions, we will, for simplicity, omit the phrase ``over \K''  in the following definitions.\\[1ex]
We define the  $S$-module of the $\K$-linear derivations of $S$ by:
\[
  \Der_\K(S,S):=\Set{\partial\in\Hom_\K(S,S) \sodass \partial(f g) = f \partial(g) + \partial(f) g,~~\forall~f,g \in S}.
\]
For any ideal $I\subseteq S$, the \deff{first cotangent cohomology module} $T^1(S/I)$ is the cokernel of the natural map $\Der_\K(S,S)\too\Hom_S(I,S/I)$.
To define $T^2(S/I)$, consider the first syzygy module of the monomial ideal $I$. That is, assuming $I$ is minimally generated by $m$ monomials, consider the kernel of the map $d$, which takes the canonical basis of $S^m$ to the unique set of minimal monomial generators of  $I$:
\[
  \begin{tikzcd}[column sep = small, row sep = 3em]
    0\ar[r]&\Syz(I)\ar[r]& S^m\ar[r]{}{d}&S\ar[r]&S/I\ar[r]&0.
  \end{tikzcd}
\]
Consider also the $S$-submodule $\Kosz(I)\subseteq \Syz(I)$ generated by the Koszul relations:
\[
  \Kosz(I) = \langle d(f)\cdot g - d(g)\cdot f \sodass f,g\in S^m\rangle.
\]
The quotient $S$-module $\Syz(I)/\Kosz(I)$ has an $S/I$-module structure.
The \deff{second cotangent cohomology module} $T^2(S/I)$ is the cokernel of the induced map
\[
  \Hom_S(S^m,S/I) \too \Hom_{S/I}\left(\frac{\Syz(I)}{\Kosz(I)},S/I\right).
\]

\subsection{Cotangent cohomology for simplicial complexes}

With the field \K\ fixed, we will denote simply by $T^{i}(\D)$ the cotangent cohomology of $S/I_\D$.
As $I_\D\subseteq S$ is a monomial ideal, the Stanley-Reisner ring $\K[\D]$, its resolution, and all the modules defined above are $\Z^n$-graded.
For $\bfc\in\Z^n$ and $i=1,2$, we denote the $\Z^n$-graded components of the cotangent cohomology modules by
\[
  T^i_\bfc(\Delta).
\]
The support of $\bfa=(a_1,\dots,a_n)\in\N^n$ is defined as the set $\supp\bfa = \Set{i\in[n]\sodass a_i\neq 0}\subseteq [n]$. Every  $\bfc\in\Z^n$ has a unique decomposition as
\[
  \bfc=\bfa-\bfb \quad\text{with ~~$\bfa, \bfb\in\N^n$ ~~and~~ $\supp \bfa \cap \supp \bfb = \emptyset$.}
\]
We paraphrase the following result. Note that a vector $\bfb\in\{0,1\}^n$ is uniquely determined by its support. 
\begin{lemma}[{\cite[Lemma 2]{altmann2000stanleyreisner}}]
  \label{lem:L2AC}
  The modules $T^i_{\bfa-\bfb}$ vanish unless $0\neq\bfb\in\{0,1\}^n$, $\supp\bfa\in\D$ and $\supp \bfb \subseteq [\lk_\D \supp \bfa]$\footnote{~Where $[\D]=\Set{v\in[n]\sodass v\in\D}$ denotes the set of vertices appearing in $\D$.}.
With these conditions  fulfilled, $T^i_{\bfa-\bfb}$  depends only on $\supp\bfa$ and \bfb.
\end{lemma}
 \noindent Furthermore, in \cite[Proposition~11]{altmann2000stanleyreisner} it is shown that  a combinatorial interpretation for the case $\bfa=0$ is enough. In particular, if  $A=\supp\bfa$, then we have for $i=1, 2$ that
\begin{equation}
  \label{eq:TiOfLink}
  T^i_{\bfa - \bfb}(\D) = T^i_{-\bfb}(\lk_\D A).
\end{equation}
\begin{convention}
  Throughout this paper $\bfb$ will always  denote a 0-1 vector, and we will use the same notation for its support. So, according to context, we may have 
  \[
    \bfb\in\set{0,1}^n \quad\text{or}\quad\bfb\subseteq [n].
    \]    
\end{convention}
\noindent To present the combinatorial characterization of $T^i_{-\bfb}(\D)$ from \cite{altmann2000stanleyreisner} we need to define two sets:
\begin{definition}
  \label{def:NandNtilde}
Let $\Delta$ be a simplicial complex on $[n]$ and $\bfb \subseteq [n]$. We define
\begin{align*}
    & N_{\bfb}(\D) := \Set{F  \in \D \sodass F \cap \bfb = \emptyset \text{~and~} F \cup \bfb \notin \D}  \text{ and }\\
    & \wt{N}_{\bfb}(\D) := \Set{F  \in N_{\bfb}(\D) \sodass \exists~ \bfb' \subsetneq \bfb \text{ with } F\cup \bfb' \notin \D}.
\end{align*}
\end{definition}
\begin{remark}
  \label{rem:Ndel}
  The above  definition implies
  \[
    N_{\bfb}(\D)=
    \begin{cases}
      \D\sdif \str_\D \bfb&\text{if~}\bfb\in\D,\\
      \D\sdif \bfb &\text{if~}\bfb\notin\D.      
    \end{cases}\\
  \]
\end{remark}
\noindent For every nonempty set $F\subset [n]$ one assigns the \defi{relatively open simplex} $\opensimplex{F}\subseteq\R^n$ as
 \[
   \textstyle
  \opensimplex{F} := \Set{\alpha:[n]\too[0,1]\sodass \sum_{i=1}^n\alpha(i)=1 \text{~and~} (\,\alpha(i)\neq 0 \iff i\in F\,)}.
\]
Each collection of subsets $\Gamma\subseteq 2^{[n]}$ thus determines a topological space in the following way:
\[
  \opensimplex{\Gamma}=
  \begin{cases}
\textstyle    \bigcup_{F\in\Gamma} \opensimplex{F}&\text{if~}\emptyset\notin \Gamma,\\[1ex]
\textstyle    \cone\left(\bigcup_{F\in\Gamma}\opensimplex{F}\right)&\text{if~}\emptyset\in \Gamma.\\    
  \end{cases}
\]
The ``usual geometric representation'' of a simplicial complex \D\ is  $\opensimplex{\D\sdif\Set{\emptyset}}$.\\[1ex]
Many of our proofs rely  on the following theorem of Altmann and Christophersen.
\begin{theorem}[{\cite[Theorem 9]{altmann2000stanleyreisner}}]
  \label{thm:TiAsRelativeCohomology}
  Let \D\ be a simplicial complex on $[n]$ and $\bfb\in\Set{0,1}^n$, which we will identify  with its support. If $\card{\bfb}>1$, then $T^i_{-\bfb}(\D)$ is given by the following relative cohomology modules
  \[
    T^i_{-\bfb}(\D) \simeq H^{i-1}(\opensimplex{N_{\bfb}(\D)},\opensimplex{\wt{N}_{\bfb}(\D)},\K)\qquad\text{for }i=1, 2.
  \]
If $\# \bfb=1$, then the above formula holds if we use the reduced relative cohomology instead.   
\end{theorem}
\noindent As a consequence, we can compute $T^1_{-\bfb}$ and $T^2_{-\bfb}$ for $\# \bfb>1$ from the following long exact sequence:
\begin{equation}
    \label{eq:les}
      \begin{tikzcd}[column sep = small, row sep = 3em]
        0\ar[r]&T^1_{-\bfb}(\D)\ar[r]&H^0\braket{N_{\bfb}}\ar[r]&H^0\braket{\wt{N}_{\bfb}}\ar[r]&T^2_{-\bfb}(\D)\ar[r]&H^1\braket{N_{\bfb}}\ar[r]&H^1\braket{\wt{N}_{\bfb}} ,
\end{tikzcd}    
\end{equation}
where we write $H^i\braket{ \bullet }$ instead of  $H^i(\braket{ \bullet })$ to improve readability.

\subsection{Remarks on $T^2(\D)$}
The following easy remarks will be used repeatedly, so it is convenient to state and to prove them.

\begin{remark}
  \label{rem:T2ei}
  For every simplicial complex $\D$ and each canonical basis vector $e_i$, we have
\[
  T^2_{-e_i}(\D) \simeq H^1\braket{N_{e_i}}.
\]
In particular, if $\Set{i}\notin \D$, then $\braket{N_{e_i}}$ is a cone, thus $T^2_{-e_i}(\D)=0$.
\end{remark}
\begin{proof}
When $\#\bfb=1$, that is when $\bfb = e_i$ for some $i$, we just need to replace cohomology with reduced cohomology in (\ref{eq:les}). The remark follows, because the only proper subset of a one-element set is the empty set, which implies $\wt{N}_{e_i}=\emptyset$. 
\end{proof}
\begin{remark}
  \label{rem:nonMinimalFace}
  If $\bfb\notin\D$ and $\bfb\notin\CC_\D$, then $T^2_{-\bfb}(\D)=0$.  
\end{remark}
\begin{proof}
  This is a direct consequence of Theorem~\ref{thm:TiAsRelativeCohomology} and the fact that $N_\bfb=\wt{N}_\bfb$ in this case.
\end{proof}
\begin{remark}
  \label{rem:T2forCircuits}
  If $\bfb\in\CC_\D$ and if $\wt{H}^0$ denotes \emph{reduced} cohomology, then
  \begin{equation}
    \label{eq:T2circuit}
    \dim_\K T^2_{-\bfb}(\D)= \dim_\K \wt{H}^0\braket{\wt{N}_\bfb}.%\max\Set{\dim_\K H^0\braket{\wt{N}_\bfb}-1,~0}.    
  \end{equation}
\end{remark}
\begin{proof} The case $\bfb=\{i\}\notin \D$ was covered in the second part of Remark~\ref{rem:T2ei}, so we may assume $\#\bfb\geqslant 2$.
  From $\bfb\in\CC_\D$ we get $N_\bfb=\D\sdif \bfb$ (cf. Remark~\ref{rem:Ndel}).
  In particular $\emptyset\in\D$,  so by definition $\braket{N_\bfb}$ is a cone.
  This means, that
  \[
    H^0\braket{N_\bfb}\simeq
    \begin{cases}
      \K&\text{if~}\D\sdif\bfb\neq \emptyset\\
      0&\text{if otherwise}
    \end{cases}
    \qquad
    \text{and}
    \qquad
    H^1\braket{N_\bfb}=0.
  \]
  So, by (\ref{eq:les}), $T^2_{-\bfb}(\D)$ is the cokernel of the  map from $H^0\braket{N_\bfb}$ to $H^0\braket{\wt{N}_{\bfb}}$. 
  This map is the zero map only when $H^0\braket{\wt{N}_{\bfb}}=0$. In other words, this map has maximal rank. So, from~(\ref{eq:les}) we get the desired formula for the dimension of $T^2_{-\bfb}(\D)$.
\end{proof}

\begin{remark}
  \label{rem:unobstructedThenConnected}
  If $T^2(\D)=0$ for a simplicial complex \D\ of dimension at least 1, then \D\ is either connected, or it has two connected components, with one of them being an isolated vertex.
\end{remark}
\begin{proof}
  Assume there exist two simplicial complexes $\D_1$ and $\D_2$, each containing at least two vertices, such that
  \[
    \D = \D_1 \cup \D_2 \quad\text{and}\quad \D_1\cap \D_2=\Set{\emptyset}.
  \]
  In other words, no face of $\D$ contains vertices both from $\D_1$ and also from $\D_2$. Choose one vertex from each: $v_1\in\D_1$ and $v_2\in \D_2$. Our claim is that for $\bfb=\Set{v_1,v_2}$ we get $T^2_{-\bfb}(\D)\neq 0$.\\
  Indeed, by Remark~\ref{rem:T2forCircuits}, we have to show that $\braket{\wt{N_\bfb}}$ has two connected components. By  $\D_1\cap\D_2=\Set{\emptyset}$ we get that any nonempty face $F\in N_\bfb$ is either in $\D_1$ or in $\D_2$, and that  $F\cup\Set{v_1}\notin\D$ or $F\cup\Set{v_2}\notin\D$. 
  Thus $\wt{N}_\bfb =N_\bfb\sdif\Set{\emptyset}$, which means that $\braket{\wt{N}_{\bfb}}$ is the geometric realization of $\D\sdif \bfb$. The latter  is the disjoint union of the geometric realizations of the nonempty simplicial complexes $\D_1\sdif\Set{v_1}$ and $\D_2\sdif\Set{v_2}$. We thus have $\dim_\K H^0\braket{\wt{N}_\bfb}\geqslant 2$.
\end{proof}
\section{One-dimensional complexes with vanishing $T^2$}
\label{sec:lowDimensions}

It follows from (\ref{eq:TiOfLink})  that if  $T^2(\D)=0$, then $T^2$ vanishes for all links of $\D$.
For this reason, we begin by characterizing the vanishing of  $T^2$ for complexes of  dimensions zero and one.
\subsection{Characterization in dimensions zero and one}
\label{subsec:dim0and1}
The first step is straightforward, as the following lemma shows.
\begin{lemma}
    \label{lem:dim0T20}
    If $\D$ is a  zero-dimensional simplicial complex on $[n]$ with no loops\footnote{~Recall that having no loops means $\{i\}\in\D$ for all $i\in[n]$.}, then 
    \[
\dim_\K    T^2_{\bfa-\bfb}(\D)=
    \begin{cases}
      \max\{n-3,0\} & \text{~if~}\bfa=0 \text{~and~}\card{\bfb}=2,\\
      0 & \text{~if otherwise.}
    \end{cases}
  \]
  In particular, $T^2(\D)=0$  if and only if $n\leqslant 3$.
\end{lemma}
\begin{proof}
    By Lemma~\ref{lem:L2AC} it is enough to check all degrees $\bfa-\bfb$ with $\supp\bfa\in \D$ and $\bfb\in\{0,1\}^n$. \\[1ex]
    If $\card{\supp\,\bfa}=1$, then $\link_\D \supp\bfa= \{\emptyset\}$, so $T^2_{\bfa-\bfb}(\D)=0$ for all $\bfb$. %\\[1ex]
    Assume  from  now  that $\bfa=0$.\\[1ex]
    If  $\card{\bfb}=1$,  then  $\braket{N_b}$ is the usual geometric realization of $\D\sdif\bfb$. By Remark~\ref{rem:T2ei} we get $T^2_{-\bfb} = 0$.\\[1ex]
    If $\card{\bfb}=2$, then $\bfb\in\CC_\D$. By Remark~\ref{rem:T2forCircuits} we have to find the number of connected components of $\braket{\wt{N}_\bfb}$, which is a union of $n-2$ distinct points. So by~(\ref{eq:T2circuit}) we get
    \[
      \dim_\K T^2_{-\bfb}(\D) = \max\Set{n-3,0}.
    \]
    If $\card{\bfb}\geqslant 3$, then $\bfb\notin\D$ and $\bfb\notin\CC_\D$, so $T^2_{-\bfb}(\D)=0$ by Remark~\ref{rem:nonMinimalFace}.
\end{proof}
\noindent We recall the following definitions.
The \deff{local degree} of a vertex is the number of edges incident to it. 
A vertex $w\in[n]$ is a \deff{neighbour} (in \D) of the vertex $v\in[n]$, if $\Set{v,w}\in\D$.
We call the set of all neighbors of $v$ the \deff{neighborhood} of $v$ and denote it by $\nu(v)$. Note that, by our definition, if $v\in\D$, then $v\in\nu(v)$, in particular, the cardinality of $\nu(v)$ is one more than the local degree of $v$. For a subset $M$ of $[n]$ we denote by $\nu(M)$ the union of the neighborhoods of all its elements, and we say that $M$ is \deff{dominating} if
  \[
    \textstyle
   \nu(M):= \bigcup_{v\in M} \nu(v) = [n].
  \]
  For $\bfb\subseteq[n]$, we denote the union of $\bfb$ with the set of common neighbours of its elements  by
  \[
    \textstyle
    \widehat{\bfb}:=\bfb \cup \left(\bigcap_{v\in\bfb}\nu(v)\right).
  \]
  A \deff{cycle} in $\D$ is a sequence of $c$ distinct  vertices $(v_1,\dots,v_c)$, with $c\geqslant 3$, such that $\Set{v_i,v_{i+1}}\in\D$ for all $i=1,\dots,c$, where $v_{c+1}=v_1$.
\begin{theorem}
\label{thm:unobstructedDim1}
  A one-dimensional simplicial complex $\Delta$ on $[n]$ satisfies $T^2(\D)=0$ if and only if the following three conditions hold:
  \begin{enumerate}[label=(\roman*),font=\upshape, itemsep=5pt]
\item\label{item:unobstr1i} Every vertex in \D\ has local degree at most three.
\item\label{item:unobstr1ii} Every cycle in \D\ is a dominating set.
\item\label{item:unobstr1iii} For every minimal nonface $\bfb\in\CC_\D$ with $\#\bfb=2$,  the simplicial complex  $\D\sdif\widehat{\bfb}$ is connected.
\end{enumerate}
\end{theorem}
\begin{remark}
  \label{rem:circulant}
  One may think that there is a connection between the three conditions of Theorem~\ref{thm:unobstructedDim1} and circulant graphs (cf. \cite{Hos07}).
  At least in the absence of leaves this seemed plausible.
  However, several leafless complexes in Figure~\ref{fig:unobstructed1dim} (e.g. the ones with 7 vertices and 9 edges)  are  not circulant  graphs.
\end{remark}
\begin{proof} Note that each of the first two conditions is equivalent to the vanishing of $T^2$ in some multidegree:  
  Because $\dim \D=1$, we get from Lemma~\ref{lem:dim0T20} that
  \begin{equation}
    \label{eq:(i)iff}
    \text{(i)} \iff T^2_{\bfa-\bfb}(\D)=0 \text{~for all~} \bfa\in\N^n \text{~with~} \card{\supp \bfa}=1  \text{~and for all~} \bfb\in\N^n.     
  \end{equation}
 By Remark~\ref{rem:T2ei}, $T^2_{-e_i}(\D)$ vanishes if and only if $\D\sdif \nu(i)$ does not contain any cycles. This is equivalent to every cycle must intersect $\nu(i)$. Requiring this for all $i$ one gets: 
 \begin{equation}
   \label{eq:(ii)iff}
   \text{(ii)}    \iff T^2_{-e_i}(\D)=0\text{~ for all~} i\in[n].
 \end{equation}
The vanishing of $T^2_{-\bfb}(\D)$ for all $\bfb\in\CC_\D$ is equivalent by Remark~\ref{rem:T2forCircuits} to the following condition: 
 \begin{enumerate}[label=(\roman*'),font=\upshape, itemsep=5pt, start=3]
\item  For every $\bfb\in\CC_\D$, the space $\braket{\wt{N}_\bfb}$ is connected.
\end{enumerate}
This condition is not equivalent to  condition (iii) from the statement of the theorem. A counterexample can be easily found using the following equality of sets:
\begin{equation}
  \label{eq:NtildeDminusbHat}
  \textstyle
  \wt{N}_\bfb \cup\Set{\emptyset} = (\D\sdif\widehat{\bfb}) \cup \Set{F\in\D\sodass F\cap \left(\bigcap_{v\in\bfb}\nu(v)\right)\neq\emptyset}.  
\end{equation}
However, we have the following equivalence when $\dim\D=1$.\\[1ex]
\textbf{Claim:} The conditions (i), (ii), and (iii) together are equivalent to (i), (ii), and (iii') together.\\[1ex]
\textit{Proof of Claim:} For the direct implication assume (i), (ii), and that $\D\sdif\widehat{\bfb}$ is connected for all $\bfb\in\CC_\D$ with $\card\bfb=2$.
Because $\dim\D=1$, a circuit $\bfb\in\CC_\D$  has cardinality at most three.\\[1ex]
If $\card\bfb =1$, then $\wt{N}_{\bfb}$ is empty, so connected.\\[1ex]
If $\card\bfb=2$, then label it by $\bfb=\Set{1,2}$ and  assume that $\braket{\wt{N}_\bfb}$ is not connected.
This means by~(\ref{eq:NtildeDminusbHat}) that there exists at least one edge $F\in\D$ with both ends in $\nu(1)\cap\nu(2)$. Let us call these ends $v$ and $w$; i.e. $F=\Set{v,w}$. Thus we have the following  edges in \D:
\[
  \Set{v,1},\quad\Set{v,2},\quad\Set{w,1},\quad\Set{w,2},\quad\Set{v,w}.
\]
By~\ref{item:unobstr1i} there are  no further edges incident to $v$ or $w$, and at most one edge incident to each $1$ and $2$.
Because we assumed that $\braket{\wt{N}_\bfb}$ is not connected, there must be at least one further vertex $u\in\D\sdif\Set{1,2,v,w}$.
By~\ref{item:unobstr1ii} the cycles $(1,v,w)$ and $(2,v,w)$ must be dominating.
Thus this vertex $u$ must be connected to both $1$ and $2$.
But this implies that $u\in\widehat{\bfb}$, and we get $\wt{N}_\bfb=\Set{\Set{v,w}}$, which gives us a connected space -- a contradiction.\\[1ex]
If $\card\bfb =3$, then label it $\bfb=\Set{1,2,3}$. Because $\dim \D=1$, we have
\[
  \wt{N}_\bfb = \left(\D\sdif\bfb\right)\sdif\Set{\emptyset}.
\] 
So $\braket{\wt{N}_\bfb}$ is the usual geometric realization of $\D\sdif\bfb$. Our goal is thus equivalent to proving that the simplicial complex $\D\sdif\bfb$ is connected.
We distinguish two cases: $\bfb\subsetneq\widehat{\bfb}$ and $\bfb=\widehat{\bfb}$. \\[1ex]
If $\bfb\subsetneq\widehat{\bfb}$, then by (i) we have precisely one extra vertex in $\widehat{\bfb}$, call it 4.
So $\D|_{\Set{1,2,3,4}}$ is the 1-skeleton of the 3-simplex.
Again, by (i), the vertices 1,2,3, and 4 cannot be connected to any other vertices of $\D$.
But (ii)  requires that every circuit is a dominating set, which means that $\D$ has only these four vertices, thus $\D\sdif\bfb=\Set{\emptyset,\Set{4}}$ is connected.\\[1ex]
If $\bfb=\widehat{\bfb}$, assume that $\D \sdif\bfb$ is not connected.
Since $\bfb$ is a minimal nonface of cardinality three, it is also a cycle, thus by (ii)  it must be a dominating set.
Since by assumption (i) the local degree of each vertex is at most three, there can be at most three more vertices in $\D$. 
So $\D\sdif\bfb$ is a disconnected complex on at most three vertices.
This means, that there is one vertex $v\in\D$ with local degree one.
Without loss of generality assume $\Set{v,1}$ to be  the only edge incident to the vertex $v$.
We may also assume  that the second vertex in $\D\sdif\bfb$ is called $w$ and that $\Set{2,w}\in\D$.
The set $\Set{1,w}\in\CC_\D$ then violates (iii), because the vertices $v$ and $3$ get disconnected after $\Set{1,w}$ and their only common neighbor, the vertex 2, have been removed.\\[2ex]
For the reverse implication of the Claim, we prove something stronger, namely that (iii') implies (iii).
To see this, notice that the second set on the right hand side of~(\ref{eq:NtildeDminusbHat}) consist only of edges which have at least one endpoint in $\bigcap_{v\in\bfb}\nu(v)$.
That is, these edges have at most one endpoint in $\D\sdif\widehat{\bfb}$.
This means, that adding these edges does not connect any components of $\D\sdif \widehat{\bfb}$. 
So
\[
  \dim_\K H^0\braket{\wt{N}_\bfb} \geqslant \dim_\K H^0(\D\sdif\widehat{\bfb}),
\]
and this shows the other implication of the Claim. \hfill$\blacksquare$\\[1ex]
We can now turn to proving the equivalence in the theorem.\\[1ex]
\fbox{$\then$} If $T^2(\D)=0$, then we have already seen, that it implies (i) and (ii). By Remark~\ref{rem:T2forCircuits} we also have (iii'). So the Claim that we proved implies (i), (ii), and (iii). \\[1ex]
\fbox{$\back$} The Claim implies that (i), (ii), and (iii') hold and we will use this to show that $T^2_{\bfa-\bfb}(\D)=0$ for all $\bfa$ and $\bfb$ for which this needs to be checked according to Lemma~\ref{lem:L2AC}. We have the following.\\[1ex]
\textbf{Case\,1:} $\card{\supp\,\bfa}=2$.
As $\dim \D=1$, we get $\link_\D \supp \bfa = \emptyset$ or  $\Set{\emptyset}$, so $T^2_{\bfa-\bfb}(\D)=0$ for all $\bfb$.\\[1ex]
\textbf{Case\,2:} $\card{\supp\,\bfa}=1$ is given by (i), see (\ref{eq:(i)iff}).\\[1ex]
\textbf{Case\,3:} $\card{\supp\,\bfa}=0$. We split this case in two.\\[1ex]
\textbf{Case\,3.1:} $\bfb\notin\D$. This case is covered by (iii'). \\[1ex]
\textbf{Case\,3.2:} $\bfb\in\D$.
Because $\dim\D=1$, there are only two  subcases of this case and we are done.\\[1ex]
\textbf{Case\,3.2.1:} $\card{\bfb}=1$. This case is covered by (ii).\\[1ex]
\textbf{Case\,3.2.1:} $\card{\bfb}=2$.
Label it $\bfb=\Set{1,2}$. Because $\bfb\in\D$ is a facet, we have
\[
  N_\bfb= (\D\sdif\bfb)\sdif \{\emptyset\}\quad\text{and}\quad
  \wt{N}_\bfb= \big((\D\sdif\bfb)\sdif \{\emptyset\}\big) \sdif \left(\nu(1)\cap\nu(2)\right).
\]
That means, that $\wt{N}_\bfb$ is obtained by removing some vertices from $N_\bfb$.
These vertices are precisely the common neighbors of $1$ and $2$, and thus have local degree at least two.  
From (i) we get that their local degree is at most three.
Thus, if $u\in N_\bfb \sdif \wt{N}_\bfb$, then $u$ is either isolated in $N_\bfb$, when its local degree is two, or a leaf vertex in $N_\bfb$, when its local degree is three.
In both cases the map $H^0\braket{ N_\bfb}\too H^0\braket{ \wt{N}_\bfb}$ in \eqref{eq:les} is surjective, so the map $H^0\braket{ \wt{N}_\bfb}\too T^2_{-\bfb}(\D)$ is the zero map.
This implies that $T^2_{-\bfb}(\D)$ is isomorphic to the kernel of the map $H^1\braket{ N_\bfb}\too H^1\braket{ \wt{N}_\bfb}$. As only leaf vertices or isolated vertices get removed from $\braket{N_\bfb}$, this map is an isomorphism and we conclude.
 \end{proof}

\subsection{Complete classification in dimension one}
\label{subsec:fullListDim1}
In the last part of this section we will classify all one-dimensional complexes with vanishing $T^2$.
We start by looking at \deff{trees}, that is graphs without any cycles, and, more generally, graphs with leaves.
A \deff{leaf} is a vertex of local degree one. 
\begin{remark}
  \label{rem:unobstructedWithLeafs}
  If \D\ is one-dimensional, $T^2(\D)=0$, and \D\ contains a leaf $v$, then \D\ has at most five vertices.
  Furthermore:
  \begin{enumerate}[label=(\alph*),font=\upshape, itemsep=5pt]
  \item If \D\ has five vertices, then \D\ is a square with a leaf attached.
  \item A tree has vanishing $T^2$ if and only if it has at most four vertices.
  \end{enumerate}  
\end{remark}
\begin{proof}
  Let $\Set{v,w}$ be the only edge of \D\ containing $v$.
  If $w$ is connected to all vertices of \D, then, as its local degree is at most three, \D\ has at most four vertices.
  So assume there exists $u\in\D$ such that $\bfb=\Set{w,u}\notin\D$.
  By Theorem~\ref{thm:unobstructedDim1}~\ref{item:unobstr1iii} we must have $\D\sdif\widehat{\bfb}$ connected.
  As $v$ is an isolated vertex of $\D\sdif\widehat{\bfb}$, we must have $\widehat{\bfb}=[n]\sdif\Set{v}$. 
  As the local degree of $w$ can be at most three and $\Set{w,v}\in\D$,
  $w$ can have at most two common neighbors with $u$.
  This means that $\card{\widehat{\bfb}}\leqslant 4$, which  implies that \D\ has at most five vertices.
  If $\card{\widehat{\bfb}}=4$, then $w$ and $u$ must be opposite corners of a square,
  so \D\ is  a square with a leaf.
  A direct check shows the statement about trees.
\end{proof}

In the remaining part, as the case when \D\ contains a leaf has been covered by \Cref{rem:unobstructedWithLeafs}, we will consider only leafless complexes with vanishing $T^2$.
Thus we assume from now on that every \D\  contains at least one cycle.
Without loss of generality, we assume that $\gamma=(1,2,\dots,c)$ is a chordless cycle in the one-dimensional simplicial complex \D\ on $[n]$, where  $3\leqslant c \leqslant n$.
Here, \deff{chordless} means that the only edges in $\D|_{[c]}$ are $\Set{1,2},\dots,\Set{c,1}$.
In other words, there are no diagonals (chords) present.
We will  need the following brief remark. 

%\begin{remark}
 % \label{rem:trivialAboutDominating}
  %The subset $M\subset[n]$ dominates \D\ if and only if $\nu(M)\cap A=A$ for all $A\subseteq[n]$.
%\end{remark}

\begin{remark}
  \label{rem:minimalCycle}
  If \D\ is a one-dimensional simplicial complex on $[n]$, containing a cycle of length $c$, and satisfying $T^2(\D)=0$,
  then \D\ has at most $2c$ vertices.
\end{remark}

In \Cref{lem:c7goodneighbors}, below, the hypothesis $c \geqslant 7$ is sharp, as demonstrated by the classification in Figure~\ref{fig:unobstructed1dim}.

\begin{lemma}
  \label{lem:c7goodneighbors}
  If $T^2(\D)=0$ and $c\geqslant 7$, then there exists no vertex $v\in \Set{c+1,\dots,n}$ with
  \[
    \Set{i,v}, \Set{v,j} \in\D \text{~~for some~~} 1\leqslant i<j\leqslant c.
  \]  
\end{lemma}
\begin{proof}
  Assume there exists a vertex $v$ with the above property.
  By Remark~\ref{rem:minimalCycle} \D\ contains no triangles, so we may assume without loss of generality that $i=1$ and $2<j<c$.
  The existence of $v$ implies the existence of two more cycles in \D:
  \[
    \gamma_1 = (1,\dots,j,v) \quad\text{and}\quad\gamma_2 = (j,\dots,c,1,v).
  \]
  According to Theorem~\ref{thm:unobstructedDim1}, every cycle in \D\ must be a dominating set,
  and the local degree of any vertex is at most three.
  So, besides $1$ and $j$, there can be at most one more vertex $u\in\D$ with $\Set{v,u}\in\D$.
  Because $\gamma_1$ and $\gamma_2$ are dominating, we have  that
  \begin{eqnarray*}
    \gamma_2=\gamma_2\cap\nu(\gamma_1)  &\subseteq& \Set{c,1,j,j+1,v,u} \text{~~~and}\\
    \gamma_1=\gamma_1\cap\nu(\gamma_2)  &\subseteq& \Set{j-1,j,1,2,v,u}.
  \end{eqnarray*}
  This implies, that $\Set{1,\dots,c,v} = \gamma_1\cup\gamma_2\subseteq\Set{1,2,j-1,j,j+1,c,v,u}$.
  Because $c\geqslant 7$, we have only two possible combinations for $c$, $j$ and $u$ for which the above inclusion holds:
  \[
    (c,j,u)=(7,4,6) \text{~or~} (c,j,u)=(7,5,3).
  \]
  In the first case, the cycle $(1,v,6,7)$ is not dominating because $3\notin\nu(1,v,6,7)$. In the second case, the cycle $(1,v,3,2)$ is not dominating because $6\notin\nu(1,v,3,2)$. Thus, by Theorem~\ref{thm:unobstructedDim1}~\ref{item:unobstr1ii}, $T^2(\D) \neq 0$, a contradiction.
\end{proof}
\begin{proposition}
  \label{prop:unobstructedDim1}
  If $\D$ is a one-dimensional simplicial complex containing a chordless cycle of length $c\geqslant 7$, then $T^2(\D) \neq 0$.
\end{proposition}
\begin{proof}
  Let us assume that $T^2(\D)=0$  and that \D\ contains the chordless cycle  $\gamma=(1,2,\dots, c)$.
  Because $c\geqslant 7$,  we can choose a vertex $k\in\Set{4,\dots,c-2}$ and define $\bfb=\Set{1,k}$.
  Because $\gamma$ is chordless, we have $\bfb\in \CC_\D$.
  By Lemma~\ref{lem:c7goodneighbors} and by the choice of $k$ we have $\bfb=\widehat{\bfb}$.
  So, according to Theorem~\ref{thm:unobstructedDim1}~(iii), we  have that
  \[
    \D\setminus\bfb\text{~~ is connected.}
  \]
  This means there must be a connection between the connected sets
  \[
    \Set{2,\dots,k-1} \text{~~and~~}\Set{k+1,\dots,c}.
  \]
  Again, because $\gamma$ is chordless, this connection must involve vertices in $[n] \sdif [c]$.
  According to Lemma~\ref{lem:c7goodneighbors}, there must be at least two such vertices on the connecting path.
  By Theorem~\ref{thm:unobstructedDim1}~(ii), the cycle $\gamma$ is dominating, meaning there is a connecting path with exactly two vertices outside of $[c]$.
  Let us denote these vertices as $v$ and $w$.
  Therefore, we must have $\Set{v,w} \in \D$ and
  \begin{eqnarray*}
    \exists\,i\in\Set{2,\dots,k-1}&\text{~with~}&\Set{i,v}\in\D,\\
    \exists\,j\in\Set{k+1,\dots,c}&\text{~with~}&\Set{w,j}\in\D.
  \end{eqnarray*}
  The chordless assumption, coupled with Lemma~\ref{lem:c7goodneighbors}, implies that there are no further edges in
  $\D|_{\Set{1,\dots,c,v,w}}$, beyond   those in $\gamma$ and the three edges $\Set{i,v}$, $\Set{v,w}$, and $\Set{w,j}$.   
  However,  two more cycles exist:
  \[
    \gamma_1 = (v,i,\dots,j,w) \quad\text{and}\quad\gamma_2 = (w,j,\dots,c,1,\dots,i,v),
  \]
  which, by Theorem~\ref{thm:unobstructedDim1}~\ref{item:unobstr1ii}, must also be dominating and thus satisfy:
  \begin{eqnarray*}
    \gamma_2=\gamma_2\cap\nu(\gamma_1)  &=& \Set{i-1,i,j,j+1,v,w} \text{~~~and}\\
    \gamma_1=\gamma_1\cap\nu(\gamma_2)  &=& \Set{i,i+1,j-1,j,v,w}.
  \end{eqnarray*}
  Then
  \[
    \Set{1,\dots,c,v,w} = \gamma_1\cup\gamma_2 = \Set{i-1,i,i+1,j-1,j,j+1,v,w},  
  \]          
 while, by the assumption $c \geq 7$, \[
 \#\Set{1,\dots,c,v,w} \geq 9>8 \geq \#\Set{i-1,i,i+1,j-1,j,j+1,v,w},
 \] a contradiction.
\end{proof}
From Theorem~\ref{thm:unobstructedDim1}~(ii), Remark~\ref{rem:unobstructedWithLeafs} and Proposition~\ref{prop:unobstructedDim1} we obtain that, if $\dim\D=1$ and $T^2(\D)=0$, then \D\ has at most 12 vertices. We can improve this bound.
\begin{proposition}
  \label{prop:no9ubobstructed}
 A  one-dimensional simplicial complex \D\ with $T^2(\D)=0$ has at most eight vertices. 
\end{proposition}
\begin{proof}
  Let \D\ with $T^2(\D)=0$ be a one-dimensional simplicial complex on $[n]$.
  According to the previous comment, it suffices to prove that this leads to a contradiction when $n\in\Set{9,10,11,12}$.
  Remark~\ref{rem:unobstructedWithLeafs} and $n\geqslant 9$ imply that  \D\ does not have any leaves.
  By Theorem~\ref{thm:unobstructedDim1}~\ref{item:unobstr1ii} and Proposition~\ref{prop:unobstructedDim1},
  the chordless cycles in \D\ have to have length at least $\lceil \frac{n}{2}\rceil$.
  This means that \D\ has only chordless hexagons if $n\in\Set{11,12}$, and chordless pentagons or hexagons if $n\in\Set{9,10}$.  \\[1ex]
  \textbf{Claim.} \D\ must contain a pentagon.\\[1ex]
  \textit{Proof of Claim.} Assume \D\ contains no pentagon.
  Based on the above observations, this means that all the cycles in \D\ must be of length 6.
  We may assume $\gamma=(1,2,3,4,5,6)$ is a chordless cycle in \D. Because $n\geqslant 9$, we have $[n]\sdif[6]\neq\emptyset$. 
  As $\gamma$ is dominating, for every $v\in[n]\sdif[6]$ there exists $i_v\in[6]$ such that  $\Set{v,i_v}\in\D$.
  Because \D\ does not contain pentagons, $i_v$ must be uniquely determined:
  \[
    \nu(v)\cap [6] = \Set{i_v}.
  \]
  As \D\ cannot contain any leaves, there must be some vertex $w\in[n]\sdif[6]$ with $\Set{v,w}\in\D$.
  Because \D\ does not contain any pentagons, the shortest path from $i_v$ to $i_w$ in $\gamma$ must be of length at least three.
  But $\gamma$ is a hexagon, so there is only one such vertex. In other words, we showed that
  \[
    \forall\,v\in[n]\sdif[6],~~\exists!\,w\in[n]\sdif[6]\text{~with~}\Set{v,w}\in\D.
  \]
  This implies that condition \ref{item:unobstr1iii} from Theorem~\ref{thm:unobstructedDim1} fails for $\bfb=\Set{i_v,i_w}$, a contradiction.  \hfill$\blacksquare$\\[1ex]
The Claim  shows that  $n\in\Set{11,12}$ leads to a contradiction.
 So we can assume that $9\leqslant n\leqslant 10$ and that $\gamma=(1,2,3,4,5)$ is a chordless cycle in \D. Furthermore, as $\gamma$ dominates \D, we may assume  that
  \[
    \Set{1,6},\Set{2,7},\Set{3,8},\Set{4,9}\in\D \text{~and, if $n=10$,~}\Set{5,10}\in\D.    
  \]
  The only edges that avoid the formation of squares are
  \[
    \Set{6,8},\Set{6,9},\Set{7,9}\text{~~~and, if $n=10$,~}\Set{7,10},\Set{8,10}.
  \]
%  Some of these edges must be in \D because otherwise $6,7,8,9,10$ would be leaves.
  Because every cycle must dominate \D,  all of these edges have to be edges of \D.
  For instance, if $\Set{6,8}\in\D$, then,  as the cycle $(1,2,3,8,6)$ dominates \D, we must also have
  \( \Set{6,9}\in\D. \)      
  We have thus excluded  all but the two complexes depicted in Figure~\ref{fig:twoSurvivors}.
 \begin{figure}[H]
    \centering
    \begin{tikzpicture}[scale=1]
      \tkzDefPoints{
        0.3/0/A,
        1.5/0/B,
        1.9/1.1/C,
        0.9/1.7/D,
        -0.1/1.1/E,       
        0.6/0.4/F,
        1.15/0.4/G,
        1.3/1/H,
        0.5/1/I};
      \tkzLabelPoint[left](E){\small 1}
      \tkzLabelPoint[below left](A){\small 2}
      \tkzLabelPoint[below right](B){\small 3}
      \tkzLabelPoint[right](C){\small 4}
      \tkzLabelPoint[above](D){\small 5}
      \tkzLabelPoint[left](F){\small 7}
      \tkzLabelPoint[right](G){\small 8}
      \tkzLabelPoint[above](H){\small 9}
      \tkzLabelPoint[above](I){\small 6}
    \tkzDrawLine[add = 0 and 0,red](A,B);
    \tkzDrawLine[add = 0 and 0](B,C);
    \tkzDrawLine[add = 0 and 0](C,D);
    \tkzDrawLine[add = 0 and 0](D,E);
    \tkzDrawLine[add = 0 and 0](E,A);
    \tkzDrawLine[add = 0 and 0,red](A,F);
    \tkzDrawLine[add = 0 and 0,red](B,G);        
    \tkzDrawLine[add = 0 and 0](C,H);
    \tkzDrawLine[add = 0 and 0](E,I);
    \tkzDrawLine[add = 0 and 0,red](F,H);
    \tkzDrawLine[add = 0 and 0,red](G,I);
    \tkzDrawLine[add = 0 and 0,red](H,I);

    \fill[color=black] (A) circle (1.5pt);
    \fill[color=black] (B) circle (1.5pt);
    \fill[color=black] (C) circle (1.5pt);
    \fill[color=red] (D) circle (1.5pt);
    \fill[color=black] (E) circle (1.5pt);
    \fill[color=black] (F) circle (1.5pt);
    \fill[color=black] (G) circle (1.5pt);
    \fill[color=black] (H) circle (1.5pt);
    \fill[color=black] (I) circle (1.5pt);
  \end{tikzpicture}
  \qquad\qquad
\begin{tikzpicture}[scale=1]
      \tkzDefPoints{
        0.3/0/A,
        1.5/0/B,
        1.9/1.1/C,
        0.9/1.7/D,
        -0.1/1.1/E,       
        0.5/0.4/F,
        1.25/0.4/G,
        1.4/1/H,
        0.4/1/I,
        0.9/1.3/J};
      \tkzLabelPoint[left](E){\small 1}
      \tkzLabelPoint[below left](A){\small 2}
      \tkzLabelPoint[below right](B){\small 3}
      \tkzLabelPoint[right](C){\small 4}
      \tkzLabelPoint[above](D){\small 5}
      \tkzLabelPoint[left](F){\small 7}
      \tkzLabelPoint[left, xshift =1pt, yshift=-3pt](G){\small 8}
      \tkzLabelPoint[below](H){\small 9}
      \tkzLabelPoint[above](I){\small 6}
      \tkzLabelPoint[right,xshift=-2pt](J){\small 10}
    \tkzDrawLine[add = 0 and 0,red](A,B);
    \tkzDrawLine[add = 0 and 0](B,C);
    \tkzDrawLine[add = 0 and 0](C,D);
    \tkzDrawLine[add = 0 and 0](D,E);
    \tkzDrawLine[add = 0 and 0](E,A);
    \tkzDrawLine[add = 0 and 0,red](A,F);
    \tkzDrawLine[add = 0 and 0,red](B,G);        
    \tkzDrawLine[add = 0 and 0](C,H);
    \tkzDrawLine[add = 0 and 0](E,I);
    \tkzDrawLine[add = 0 and 0](D,J);
    \tkzDrawLine[add = 0 and 0,red](F,H);
    \tkzDrawLine[add = 0 and 0,red](G,I);
    \tkzDrawLine[add = 0 and 0,red](H,I);
    \tkzDrawLine[add = 0 and 0](F,J);
    \tkzDrawLine[add = 0 and 0](G,J);

    \fill[color=black] (A) circle (1.5pt);
    \fill[color=black] (B) circle (1.5pt);
    \fill[color=black] (C) circle (1.5pt);
    \fill[color=red] (D) circle (1.5pt);
    \fill[color=black] (E) circle (1.5pt);
    \fill[color=black] (F) circle (1.5pt);
    \fill[color=black] (G) circle (1.5pt);
    \fill[color=black] (H) circle (1.5pt);
    \fill[color=black] (I) circle (1.5pt);
    \fill[color=black] (J) circle (1.5pt);
  \end{tikzpicture}  
  \caption[Surviving candidates]{The only two surviving candidates  on at least 9 vertices.}
  \label{fig:twoSurvivors}
  \end{figure}
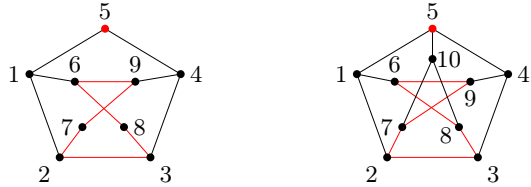
\noindent However, for both of these candidates, the cycle $(2,7,9,6,8,3)$ does not intersect the neighborhood of the vertex 5. Thus, by Theorem~\ref{thm:unobstructedDim1}, they are have nonvanishing $T^2$.
\end{proof}
Not many candidates are left for one-dimensional simplicial complexes with vanishing $T^2$.
According to Lemma~\ref{lem:dim0T20}, complete bipartite graphs $K_{r,s}$ satisfy $T^2(K_{r,s})=0$  precisely when  $1\leqslant r,s\leqslant 3$.
This follows because these complexes are joins of zero-dimensional complexes and because $T^i$ of a join is obtained via tensor product from the $T^i$s of the joined complexes \cite[Proposition\,2.3]{CI14} (see also~(\ref{eq:TofJoins}) below).
Thus, $T^2$ of the join of two complexes vanishes if and only if $T^2$ vanishes for both joined complexes. \\[1ex]
However, there are more  one-dimensional complexes with $T^2=0$ beyond these six. 

\begin{theorem}
    \label{thm:summarySection3}
    There are precisely 26 one dimensional simplicial complexes with $T^2=0$, all illustrated in Figure~\ref{fig:unobstructed1dim}. 
    In particular, the maximal possible number of vertices is eight, and,  if a leaf is present, then there are at most five vertices.    
\end{theorem}
\begin{proof}    
Besides verifying that the three conditions of Theorem~\ref{thm:unobstructedDim1} are fulfilled by all 26 complexes presented in Figure~\ref{fig:unobstructed1dim}, one may also quickly check this using the Macaulay2 package \texttt{VersalDeformations} \cite{M2,Ilt12} to confirm that $T^2=0$ for all these complexes.\\[1ex]
To establish that these are the only one-dimensional complexes with vanishing $T^2$, first note that the cases where \D\ contains a leaf are fully covered in Remark~\ref{rem:unobstructedWithLeafs}.
In the absence of leaves, one has to start from the shortest cycle contained in \D.
Using the minimality of this cycle and the fact that its vertices form a dominating set for \D, one obtains the list in Figure~\ref{fig:unobstructed1dim} following similar arguments as those used in the proof of  Proposition~\ref{prop:no9ubobstructed}.
\end{proof}
\begin{figure}[H]
  \centering

  \begin{tabular}{llllll}
    \begin{tikzpicture}[scale=.5]
    \tkzDefPoints{
      1/1/A,
      2/1/B};
    \fill[color=blue] (A) circle (3pt);
    \fill[color=blue] (B) circle (3pt);
    \tkzDrawLine[color = blue, add = 0 and 0](A,B);    
  \end{tikzpicture}
  &&
    
    %%% WITH 3 VERTICES:
  \begin{tikzpicture}[scale=.5]
    \tkzDefPoints{
      1/1/A,
      1/2/B,
    2/1.5/C};
    \fill[color=black] (A) circle (3pt);
    \fill[color=black] (B) circle (3pt);
    \fill[color=black] (C) circle (3pt);
    \tkzDrawLine[add = 0 and 0](A,B);    
  \end{tikzpicture}
  &
  \begin{tikzpicture}[scale=.5]
    \tkzDefPoints{
      1/1/A,
      1/2/B,
    2/1.5/C};
    \fill[color=blue] (A) circle (3pt);
    \fill[color=blue] (B) circle (3pt);
    \fill[color=blue] (C) circle (3pt);
    \tkzDrawLine[color=blue, add = 0 and 0](A,B);
    \tkzDrawLine[color=blue, add = 0 and 0](A,C);    
  \end{tikzpicture}
  &
  \begin{tikzpicture}[scale=.5]
    \tkzDefPoints{
      1/1/A,
      1/2/B,
    2/1.5/C};
    \fill[color=blue] (A) circle (3pt);
    \fill[color=blue] (B) circle (3pt);
    \fill[color=blue] (C) circle (3pt);
    \tkzDrawLine[color=blue, add = 0 and 0](A,B);
    \tkzDrawLine[color=blue, add = 0 and 0](A,C);
    \tkzDrawLine[color=blue, add = 0 and 0](B,C);    
  \end{tikzpicture}\\[3ex]
    
   %%% WITH 4 VERTICES:
  \begin{tikzpicture}[scale=.5]
    \tkzDefPoints{
      1/1/A,
      2/1/B,
      2/2/C,
      1/2/D};
    \fill[color=blue] (A) circle (3pt);
    \fill[color=blue] (B) circle (3pt);
    \fill[color=blue] (C) circle (3pt);
    \fill[color=blue] (D) circle (3pt);
    \tkzDrawLine[color=blue, add = 0 and 0](A,B);
    \tkzDrawLine[color=blue, add = 0 and 0](A,C);
    \tkzDrawLine[color=blue, add = 0 and 0](A,D);    
  \end{tikzpicture}
    &
  \begin{tikzpicture}[scale=.5]
    \tkzDefPoints{
      1/1/A,
      2/1/B,
      2/2/C,
      1/2/D};
    \fill[color=black] (A) circle (3pt);
    \fill[color=black] (B) circle (3pt);
    \fill[color=black] (C) circle (3pt);
    \fill[color=black] (D) circle (3pt);
    \tkzDrawLine[add = 0 and 0](A,B);
    \tkzDrawLine[add = 0 and 0](B,C);
    \tkzDrawLine[add = 0 and 0](A,D);    
  \end{tikzpicture}
  &
  \begin{tikzpicture}[scale=.5]
    \tkzDefPoints{
      1/1/A,
      2/1/B,
      2/2/C,
      1/2/D};
    \fill[color=blue] (A) circle (3pt);
    \fill[color=blue] (B) circle (3pt);
    \fill[color=blue] (C) circle (3pt);
    \fill[color=blue] (D) circle (3pt);
    \tkzDrawLine[color=blue, add = 0 and 0](A,B);
    \tkzDrawLine[color=blue, add = 0 and 0](B,C);
    \tkzDrawLine[color=blue, add = 0 and 0](C,D);
    \tkzDrawLine[color=blue, add = 0 and 0](D,A);    
  \end{tikzpicture}
  &
  \begin{tikzpicture}[scale=.5]
    \tkzDefPoints{
      1/1/A,
      2/1/B,
      2/2/C,
      1/2/D};
    \fill[color=black] (A) circle (3pt);
    \fill[color=black] (B) circle (3pt);
    \fill[color=black] (C) circle (3pt);
    \fill[color=black] (D) circle (3pt);
    \tkzDrawLine[add = 0 and 0](A,B);
    \tkzDrawLine[add = 0 and 0](B,D);
    \tkzDrawLine[add = 0 and 0](C,D);
    \tkzDrawLine[add = 0 and 0](D,A);    
  \end{tikzpicture}
    &
  \begin{tikzpicture}[scale=.5]
    \tkzDefPoints{
      1/1/A,
      2/1/B,
      2/2/C,
      1/2/D};
    \fill[color=blue] (A) circle (3pt);
    \fill[color=blue] (B) circle (3pt);
    \fill[color=blue] (C) circle (3pt);
    \fill[color=blue] (D) circle (3pt);
    \tkzDrawLine[color=blue, add = 0 and 0](A,B);
    \tkzDrawLine[color=blue, add = 0 and 0](B,C);
    \tkzDrawLine[color=blue, add = 0 and 0](C,D);
    \tkzDrawLine[color=blue, add = 0 and 0](D,A);
    \tkzDrawLine[color=blue, add = 0 and 0](D,B);    
  \end{tikzpicture}
  &
  \begin{tikzpicture}[scale=.5]
    \tkzDefPoints{
      1/1/A,
      2/1/B,
      2/2/C,
      1/2/D};
    \fill[color=blue] (A) circle (3pt);
    \fill[color=blue] (B) circle (3pt);
    \fill[color=blue] (C) circle (3pt);
    \fill[color=blue] (D) circle (3pt);
    \tkzDrawLine[color=blue, add = 0 and 0](A,B);
    \tkzDrawLine[color=blue, add = 0 and 0](B,C);
    \tkzDrawLine[color=blue, add = 0 and 0](C,D);
    \tkzDrawLine[color=blue, add = 0 and 0](D,A);
    \tkzDrawLine[color=blue, add = 0 and 0](D,B);
    \tkzDrawLine[color=blue, add = 0 and 0](A,C);    
  \end{tikzpicture}\\[3ex]
   
  %%% WITH 5 VERTICES:
  \begin{tikzpicture}[scale=.5]
    \tkzDefPoints{
      1/1/A,
      2/1/B,
      2.3/1.8/C,
      1.5/2.4/D,
      0.7/1.8/E};
    \fill[color=black] (A) circle (3pt);
    \fill[color=black] (B) circle (3pt);
    \fill[color=black] (C) circle (3pt);
    \fill[color=black] (D) circle (3pt);
    \fill[color=black] (E) circle (3pt);
    \tkzDrawLine[add = 0 and 0](A,B);
    \tkzDrawLine[add = 0 and 0](B,C);
    \tkzDrawLine[add = 0 and 0](C,E);
    \tkzDrawLine[add = 0 and 0](E,A);
    \tkzDrawLine[add = 0 and 0](D,E);
  \end{tikzpicture}
  &
 \begin{tikzpicture}[scale=.5]
    \tkzDefPoints{
      1/1/A,
      2/1/B,
      2.3/1.8/C,
      1.5/2.4/D,
      0.7/1.8/E};
    \fill[color=black] (A) circle (3pt);
    \fill[color=black] (B) circle (3pt);
    \fill[color=black] (C) circle (3pt);
    \fill[color=black] (D) circle (3pt);
    \fill[color=black] (E) circle (3pt);
    \tkzDrawLine[add = 0 and 0](A,B);
    \tkzDrawLine[add = 0 and 0](B,C);
    \tkzDrawLine[add = 0 and 0](C,E);
    \tkzDrawLine[add = 0 and 0](E,A);
    \tkzDrawLine[add = 0 and 0](D,E);
    \tkzDrawLine[add = 0 and 0](D,C);
  \end{tikzpicture}
    &
  \begin{tikzpicture}[scale=.5]
    \tkzDefPoints{
      1/1/A,
      2/1/B,
      2.3/1.8/C,
      1.5/2.4/D,
      0.7/1.8/E};
    \fill[color=black] (A) circle (3pt);
    \fill[color=black] (B) circle (3pt);
    \fill[color=black] (C) circle (3pt);
    \fill[color=black] (D) circle (3pt);
    \fill[color=black] (E) circle (3pt);
    \tkzDrawLine[add = 0 and 0](A,B);
    \tkzDrawLine[add = 0 and 0](B,C);
    \tkzDrawLine[add = 0 and 0](C,E);
    \tkzDrawLine[add = 0 and 0](E,A);
    \tkzDrawLine[add = 0 and 0](D,E);
    \tkzDrawLine[add = 0 and 0](D,C);
    \tkzDrawLine[add = 0 and 0](A,D);
  \end{tikzpicture}
      &
 \begin{tikzpicture}[scale=.5]
    \tkzDefPoints{
      1/1/A,
      2/1/B,
      2.3/1.8/C,
      1.5/2.4/D,
      0.7/1.8/E};
    \fill[color=blue] (A) circle (3pt);
    \fill[color=blue] (B) circle (3pt);
    \fill[color=blue] (C) circle (3pt);
    \fill[color=blue] (D) circle (3pt);
    \fill[color=blue] (E) circle (3pt);
    \tkzDrawLine[color=blue, add = 0 and 0](A,B);
    \tkzDrawLine[color=blue, add = 0 and 0](B,C);
    \tkzDrawLine[color=blue, add = 0 and 0](C,E);
    \tkzDrawLine[color=blue, add = 0 and 0](E,A);
    \tkzDrawLine[color=blue, add = 0 and 0](D,A);
    \tkzDrawLine[color=blue, add = 0 and 0](C,D);
  \end{tikzpicture}
        &
 \begin{tikzpicture}[scale=.5]
    \tkzDefPoints{
      1/1/A,
      2/1/B,
      2.3/1.8/C,
      1.5/2.4/D,
      0.7/1.8/E};
    \fill[color=black] (A) circle (3pt);
    \fill[color=black] (B) circle (3pt);
    \fill[color=black] (C) circle (3pt);
    \fill[color=black] (D) circle (3pt);
    \fill[color=black] (E) circle (3pt);
    \tkzDrawLine[add = 0 and 0](A,B);
    \tkzDrawLine[add = 0 and 0](B,C);
    \tkzDrawLine[add = 0 and 0](C,D);
    \tkzDrawLine[add = 0 and 0](D,E);
    \tkzDrawLine[add = 0 and 0](E,A);
  \end{tikzpicture}  \\[3ex]

   %%% WITH 6 VERTICES:
  \begin{tikzpicture}[scale=.5]
    \tkzDefPoints{
      1.1/1/A,
      1.9/1/B,
      2.3/1.7/C,
      1.9/2.4/D,
      1.1/2.4/E,
      0.7/1.7/F};
    \fill[color=black] (A) circle (3pt);
    \fill[color=black] (B) circle (3pt);
    \fill[color=black] (C) circle (3pt);
    \fill[color=black] (D) circle (3pt);
    \fill[color=black] (E) circle (3pt);
    \fill[color=black] (F) circle (3pt);
    \tkzDrawLine[add = 0 and 0](A,B);
    \tkzDrawLine[add = 0 and 0](B,C);
    \tkzDrawLine[add = 0 and 0](C,D);
    \tkzDrawLine[add = 0 and 0](D,E);
    \tkzDrawLine[add = 0 and 0](E,F);
    \tkzDrawLine[add = 0 and 0](F,A);
    \tkzDrawLine[add = 0 and 0](F,C);
    \tkzDrawLine[add = 0 and 0](A,E);
  \end{tikzpicture}
  &
  \begin{tikzpicture}[scale=.5]
    \tkzDefPoints{
      1.1/1/A,
      1.9/1/B,
      2.3/1.7/C,
      1.9/2.4/D,
      1.1/2.4/E,
      0.7/1.7/F};
    \fill[color=black] (A) circle (3pt);
    \fill[color=black] (B) circle (3pt);
    \fill[color=black] (C) circle (3pt);
    \fill[color=black] (D) circle (3pt);
    \fill[color=black] (E) circle (3pt);
    \fill[color=black] (F) circle (3pt);
    \tkzDrawLine[add = 0 and 0](A,B);
    \tkzDrawLine[add = 0 and 0](B,C);
    \tkzDrawLine[add = 0 and 0](C,D);
    \tkzDrawLine[add = 0 and 0](D,E);
    \tkzDrawLine[add = 0 and 0](E,F);
    \tkzDrawLine[add = 0 and 0](F,A);
    \tkzDrawLine[add = 0 and 0](F,C);
    \tkzDrawLine[add = 0 and 0](A,E);
    \tkzDrawLine[add = 0 and 0](B,D);
  \end{tikzpicture}
    &
  \begin{tikzpicture}[scale=.5]
    \tkzDefPoints{
      1.1/1/A,
      1.9/1/B,
      2.3/1.7/C,
      1.9/2.4/D,
      1.1/2.4/E,
      0.7/1.7/F};
    \fill[color=black] (A) circle (3pt);
    \fill[color=black] (B) circle (3pt);
    \fill[color=black] (C) circle (3pt);
    \fill[color=black] (D) circle (3pt);
    \fill[color=black] (E) circle (3pt);
    \fill[color=black] (F) circle (3pt);
    \tkzDrawLine[add = 0 and 0](A,B);
    \tkzDrawLine[add = 0 and 0](B,C);
%   \tkzDrawLine[add = 0 and 0](C,D);
    \tkzDrawLine[add = 0 and 0](D,E);
%   \tkzDrawLine[add = 0 and 0](E,F);
    \tkzDrawLine[add = 0 and 0](F,A);
    \tkzDrawLine[add = 0 and 0](F,C);
    \tkzDrawLine[add = 0 and 0](A,E);
    \tkzDrawLine[add = 0 and 0](B,D);
  \end{tikzpicture}
      &
  \begin{tikzpicture}[scale=.5]
    \tkzDefPoints{
      1.1/1/A,
      1.9/1/B,
      2.3/1.7/C,
      1.9/2.4/D,
      1.1/2.4/E,
      0.7/1.7/F};
    \fill[color=black] (A) circle (3pt);
    \fill[color=black] (B) circle (3pt);
    \fill[color=black] (C) circle (3pt);
    \fill[color=black] (D) circle (3pt);
    \fill[color=black] (E) circle (3pt);
    \fill[color=black] (F) circle (3pt);
    \tkzDrawLine[add = 0 and 0](A,B);
    \tkzDrawLine[add = 0 and 0](B,C);
%   \tkzDrawLine[add = 0 and 0](C,D);
    \tkzDrawLine[add = 0 and 0](D,E);
    \tkzDrawLine[add = 0 and 0](E,F);
%   \tkzDrawLine[add = 0 and 0](F,A);
    \tkzDrawLine[add = 0 and 0](F,C);
    \tkzDrawLine[add = 0 and 0](A,E);
    \tkzDrawLine[add = 0 and 0](B,D);
  \end{tikzpicture}
        &
   \begin{tikzpicture}[scale=.5]
    \tkzDefPoints{
      1.1/1/A,
      1.9/1/B,
      2.3/1.7/C,
      1.9/2.4/D,
      1.1/2.4/E,
      0.7/1.7/F};
    \fill[color=black] (A) circle (3pt);
    \fill[color=black] (B) circle (3pt);
    \fill[color=black] (C) circle (3pt);
    \fill[color=black] (D) circle (3pt);
    \fill[color=black] (E) circle (3pt);
    \fill[color=black] (F) circle (3pt);
    \tkzDrawLine[add = 0 and 0](A,B);
    \tkzDrawLine[add = 0 and 0](B,C);
%   \tkzDrawLine[add = 0 and 0](C,D);
    \tkzDrawLine[add = 0 and 0](D,E);
%   \tkzDrawLine[add = 0 and 0](E,F);
    \tkzDrawLine[add = 0 and 0](F,A);
    \tkzDrawLine[add = 0 and 0](F,D);
    \tkzDrawLine[add = 0 and 0](A,E);
    \tkzDrawLine[add = 0 and 0](B,D);
    \tkzDrawLine[add = 0 and 0](C,E);
  \end{tikzpicture}
          &
   \begin{tikzpicture}[scale=.5]
    \tkzDefPoints{
      1.1/1/A,
      1.9/1/B,
      2.3/1.7/C,
      1.9/2.4/D,
      1.1/2.4/E,
      0.7/1.7/F};
    \fill[color=blue] (A) circle (3pt);
    \fill[color=blue] (B) circle (3pt);
    \fill[color=blue] (C) circle (3pt);
    \fill[color=blue] (D) circle (3pt);
    \fill[color=blue] (E) circle (3pt);
    \fill[color=blue] (F) circle (3pt);
    \tkzDrawLine[color=blue, add = 0 and 0](A,B);
    \tkzDrawLine[color=blue, add = 0 and 0](B,C);
%   \tkzDrawLine[add = 0 and 0](C,D);
    \tkzDrawLine[color=blue, add = 0 and 0](D,E);
%   \tkzDrawLine[add = 0 and 0](E,F);
    \tkzDrawLine[color=blue, add = 0 and 0](F,A);
    \tkzDrawLine[color=blue, add = 0 and 0](F,C);
    \tkzDrawLine[color=blue, add = 0 and 0](F,D);
    \tkzDrawLine[color=blue, add = 0 and 0](A,E);
    \tkzDrawLine[color=blue, add = 0 and 0](B,D);
    \tkzDrawLine[color=blue, add = 0 and 0](C,E);
  \end{tikzpicture} \\[3ex]
    %%% WITH 7 VERTICES:
  \begin{tikzpicture}[scale=.5]
    \tkzDefPoints{
      1.7/1/A,
      2.5/1.6/B,
      2.5/2.5/C,
      1.7/3.0/D,
      0.9/2.5/E,
      0.9/1.6/F,
      1.7/2.0/G};
    \fill[color=black] (A) circle (3pt);
    \fill[color=black] (B) circle (3pt);
    \fill[color=black] (C) circle (3pt);
    \fill[color=black] (D) circle (3pt);
    \fill[color=black] (E) circle (3pt);
    \fill[color=black] (F) circle (3pt);
    \fill[color=black] (G) circle (3pt);
    \tkzDrawLine[add = 0 and 0](A,B);
    \tkzDrawLine[add = 0 and 0](B,C);
    \tkzDrawLine[add = 0 and 0](C,D);
    \tkzDrawLine[add = 0 and 0](D,E);
    \tkzDrawLine[add = 0 and 0](E,F);
    \tkzDrawLine[add = 0 and 0](F,A);
    \tkzDrawLine[add = 0 and 0](A,G);
    \tkzDrawLine[add = 0 and 0](C,G);
    \tkzDrawLine[add = 0 and 0](E,G);
  \end{tikzpicture}
  &
  \begin{tikzpicture}[scale=.5]
    \tkzDefPoints{
      1.7/1/A,
      2.5/1.6/B,   
      2.5/2.5/C,   
      1.7/3.0/D,   
      0.9/2.5/E,   
      0.9/1.6/F,  
      1.7/2.0/G,   
      1.25/1.25/O};
    \fill[color=black] (A) circle (3pt);
    \fill[color=black] (B) circle (3pt);
    \fill[color=black] (C) circle (3pt);
    \fill[color=black] (D) circle (3pt);
    \fill[color=black] (E) circle (3pt);
    \fill[color=black] (F) circle (3pt);
    \fill[color=black] (G) circle (3pt);
    \tkzDrawLine[add = 0 and 0](A,B);
    \tkzDrawLine[add = 0 and 0](B,C);
    \tkzDrawLine[add = 0 and 0](C,D);
    \tkzDrawLine[add = 0 and 0](D,E);
    \tkzDrawLine[add = 0 and 0](E,F);
    \tkzDrawLine[add = 0 and 0](F,A);
    \tkzDrawLine[add = 0 and 0](A,G);
    \tkzDrawLine[add = 0 and 0](C,G);
    \tkzDrawArc[thick,black](O,B)(E);
  \end{tikzpicture}  
  &&
  %%% WITH 8 VERTICES:
  \begin{tikzpicture}[scale=.5]
       \tkzDefPoints{
      0.4/0.4/A,
      1.4/0.4/B,
      1.4/1.2/C,
      0.4/1.2/D,
      0/0/E,
      1.8/0/F,
      1.8/1.6/G,
      0/1.6/H};
    \fill[color=black] (A) circle (3pt);
    \fill[color=black] (B) circle (3pt);
    \fill[color=black] (C) circle (3pt);
    \fill[color=black] (D) circle (3pt);
    \fill[color=black] (E) circle (3pt);
    \fill[color=black] (F) circle (3pt);
    \fill[color=black] (G) circle (3pt);
    \fill[color=black] (H) circle (3pt);
    \tkzDrawLine[add = 0 and 0](A,B);
    \tkzDrawLine[add = 0 and 0](B,C);
    \tkzDrawLine[add = 0 and 0](C,D);
    \tkzDrawLine[add = 0 and 0](D,A);
    \tkzDrawLine[add = 0 and 0](E,F);
    \tkzDrawLine[add = 0 and 0](F,G);
    \tkzDrawLine[add = 0 and 0](G,H);
    \tkzDrawLine[add = 0 and 0](H,E);
    \tkzDrawLine[add = 0 and 0](A,E);
    \tkzDrawLine[add = 0 and 0](B,F);
    \tkzDrawLine[add = 0 and 0](C,G);
    \tkzDrawLine[add = 0 and 0](D,H);    
  \end{tikzpicture}  
  &
    \begin{tikzpicture}[scale=.5]
    \tkzDefPoints{
      0.4/0.4/A,
      1.4/0.4/B,
      1.4/1.2/C,
      0.4/1.2/D,
      0/0/E,
      1.8/0/F,
      1.8/1.6/G,
      0/1.6/H};
    \fill[color=black] (A) circle (3pt);
    \fill[color=black] (B) circle (3pt);
    \fill[color=black] (C) circle (3pt);
    \fill[color=black] (D) circle (3pt);
    \fill[color=black] (E) circle (3pt);
    \fill[color=black] (F) circle (3pt);
    \fill[color=black] (G) circle (3pt);
    \fill[color=black] (H) circle (3pt);
    \tkzDrawLine[add = 0 and 0](A,B);
    \tkzDrawLine[add = 0 and 0](B,D);
    \tkzDrawLine[add = 0 and 0](D,C);
    \tkzDrawLine[add = 0 and 0](C,A);
    \tkzDrawLine[add = 0 and 0](E,F);
    \tkzDrawLine[add = 0 and 0](F,G);
    \tkzDrawLine[add = 0 and 0](G,H);
    \tkzDrawLine[add = 0 and 0](H,E);
    \tkzDrawLine[add = 0 and 0](A,E);
    \tkzDrawLine[add = 0 and 0](B,F);
    \tkzDrawLine[add = 0 and 0](C,G);
    \tkzDrawLine[add = 0 and 0](D,H);    
  \end{tikzpicture}
    &
  \begin{tikzpicture}[scale=.5]
    \tkzDefPoints{
      0.5/0.8/A,
      1.3/0.8/B,
      0.9/1.2/C,
      0.9/1.7/D,
      0.3/0/E,
      1.5/0/F,
      1.9/1.1/G,
      -0.1/1.1/H};
    \fill[color=black] (A) circle (3pt);
    \fill[color=black] (B) circle (3pt);
    \fill[color=black] (C) circle (3pt);
    \fill[color=black] (D) circle (3pt);
    \fill[color=black] (E) circle (3pt);
    \fill[color=black] (F) circle (3pt);
    \fill[color=black] (G) circle (3pt);
    \fill[color=black] (H) circle (3pt);
    \tkzDrawLine[add = 0 and 0](E,F);
    \tkzDrawLine[add = 0 and 0](F,G);
    \tkzDrawLine[add = 0 and 0](G,D);
    \tkzDrawLine[add = 0 and 0](D,H);
    \tkzDrawLine[add = 0 and 0](H,E);
    \tkzDrawLine[add = 0 and 0](D,C);
    \tkzDrawLine[add = 0 and 0](E,A);        
    \tkzDrawLine[add = 0 and 0](F,B);
    \tkzDrawLine[add = 0 and 0](A,C);
    \tkzDrawLine[add = 0 and 0](B,C);    
  \end{tikzpicture}  
\end{tabular} 
\caption[$T^2=0$ in dimension 1]{The list of all one-dimensional simplicial complexes with vanishing $T^2$. Those simplicial complexes that are moreover matroids appear in blue. Note that all these matroids are joins of uniform matroids, with one exception: the one in the second row and fifth column.}
\label{fig:unobstructed1dim}
\end{figure}
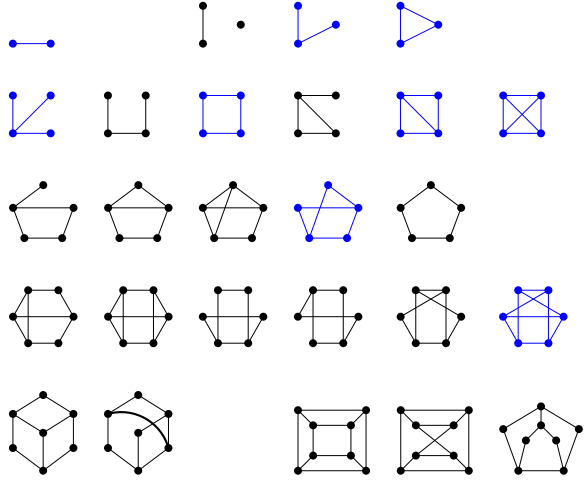

\section{$T^2$ for uniform matroids}
For every finite set $E$ and every $r\in\Z_{\geqslant 0}$, the \deff{uniform matroid} of rank $r$ on $E$ is the collection of all subsets of $E$ of cardinality at most $r$:
\[
  \CU_{E}^r = \Set{F\subseteq E \sodass \card{F} \leqslant r}.
\]
The dimensions of the multigraded components of first cotangent cohomology of $\CU_{[n]}^r$ were computed in \cite[Example~4.10]{BC2023}.
We compute now these dimensions for $T^2(\CU_{[n]}^r)$. By definition, for any face $A\in \CU_{[n]}^r$ we have
\begin{equation}
  \label{eq:linkUnr}
  \link_{\CU_{[n]}^r} A = \CU_{[n]\text{\raisebox{-.1ex}{$\sdif A$}}}^{r-\card{A}}.
\end{equation}
Thus, according to~(\ref{eq:TiOfLink}), it will be enough to determine $T^2(\CU_{[n]}^r)$  in purely negative multidegrees for all $n$ and $r$.
\begin{proposition}\label{prop:uniformT2}
For every $n\in\Z_{>0}$ and every $r\in\Z$ with $1\leqslant r\leqslant n$ we have
\[
  \dim_\K T^2_{-\bfb}(\CU_{[n]}^r) =
  \begin{cases}
    0&\textup{~if~~} \# \bfb \neq 2 \text{ or }r\geqslant n-1,\\
    r\binom{n-2}{r}-\binom{n-2}{r-1} & \textup{~if~~} \# \bfb = 2.
  \end{cases}
\]
In particular, $T^2_{-\bfb}(\CU_{[n]}^r)=0$  if and only if $r\geqslant n-2$; that is if the matroid is a full simplex, the boundary of a full simplex, or the uniform matroid $\CU_{[n]}^{n-2}$.
\end{proposition}
\begin{proof}
  To simplify notation, we  write $\CU$ for $\CU_{[n]}^r$ throughout this proof.
  By Definition~\ref{def:NandNtilde} we have
  \begin{eqnarray*}
    N_\bfb(\CU) &=& \Set{F\subseteq [n]\sdif\bfb \sodass r-\card{\bfb}<\card F \leqslant r}\quad\text{and}\\
    \wt{N}_\bfb(\CU) &=& \Set{F\subseteq [n]\sdif\bfb \sodass r-\card{\bfb}+1<\card F \leqslant r}.
  \end{eqnarray*}
  By the symmetry of $\CU$, we may assume that $\bfb=\Set{n-b+1,\dots,n}$, so $[n]\sdif\bfb=[n-b]$ and $\card{\bfb}=b$. This implies 
  \[
    N_\bfb(\CU) = \CU_{[n-b]}^{r} \sdif  \CU_{[n-b]}^{r-b}
    \qquad \text{and} \qquad
    \wt{N}_\bfb(\CU) = \CU_{[n-b]}^{r} \sdif  \CU_{[n-b]}^{r-b+1}.
  \]
  As sets partially ordered by inclusion, $N_\bfb(\CU)$ and $\wt{N}_\bfb(\CU)$ are  rank-selected subposets of the Boolean poset of all subsets of $[n-b]$.
  Since Boolean posets are shellable, and by \cite[Theorem 4.1]{Bjo} every rank selected subposet of a shellable poset is again shellable, we have that both $N_\bfb(\CU)$ and $\wt{N}_\bfb(\CU)$ are shellable posets.
  In particular, their order complexes, when non-empty, are homotopy equivalent to a wedge of spheres of dimension one smaller than their rank: $b$ and $b-1$, respectively. 
Whenever $\emptyset\notin N_\bfb(\CU)$, the order complex $N_\bfb(\CU)$ is a barycentric subdivision of $\braket{N_\bfb(\CU)}$. Hence, $\braket{N_\bfb(\CU)}$  is homotopy equivalent to wedge of $(b-1)$-spheres and, if $b>1$ holds, $\braket{\wt{N}_\bfb(\CU)}$ is homotopy equivalent to a wedge of $(b-2)$-spheres.
We recall for convenience  (\ref{eq:les}):
\begin{equation}
    \tag{\ref{eq:les}}
      \begin{tikzcd}[column sep = small, row sep = 3em]
        0\ar[r]&T^1_{-\bfb}(\D)\ar[r]&H^0\braket{N_{\bfb}}\ar[r]&H^0\braket{\wt{N}_{\bfb}}\ar[r]&T^2_{-\bfb}(\D)\ar[r]&H^1\braket{N_{\bfb}}\ar[r]&H^1\braket{\wt{N}_{\bfb}}
\end{tikzcd}    
\end{equation}
From this long exact sequence and the above observations, we  deduce the following:\\[1ex]
If {$b=1$}, then 
\(  H^1\braket{N_\bfb(\CU)}=0
   \text{ and }
  H^0\braket{\wt{N}_\bfb(\CU)}=0,
\)
thus  $T^2_{-\bfb}(\CU)=0$.\\[1ex]
If {$b>2$}, then 
\(
 H^1\braket{N_\bfb(\CU)}=0 
 \text{ and }
 H^0\braket{N_\bfb(\CU)}\cong H^0\braket{\wt{N}_\bfb(\CU)}\cong \mathbb{K},  
\)
so, because the map from  $H^0\braket{N_\bfb(\CU)}$ to $H^0\braket{\wt{N}_\bfb(\CU)}$ has maximal rank, we also get $T^2_{-\bfb}(\CU)=0$.\\[1ex]
If $\#\bfb=2$ we have that
the order complex of  $\wt{N}_\bfb(\CU)$ is the 0-dimensional complex whose vertices are the $r$-subsets of $[n]\sdif\bfb$. This means
\begin{eqnarray*}
  \dim_\K H^0\braket{\wt{N}_\bfb(\CU)}&=&{ \textstyle \binom{n-2}{r}} ,\\
  \dim_\K H^1\braket{\wt{N}_\bfb(\CU)}&=&0.
\end{eqnarray*}
   The order complex of $N_\bfb(\CU)$ is the connected 1-dimensional simplicial complex whose vertices are all $r$- and $(r-1)$-subsets of $[n]\sdif\bfb$, and whose edges are pairs of such subsets, one strictly contained in the other. From an elementary  count we obtain
   \begin{eqnarray*}
      \dim_\K H^0\braket{N_\bfb(\CU)}&=&1,\\
      \dim_\K H^1\braket{N_\bfb(\CU)}&=&\textstyle r\binom{n-2}{r}-\left(\binom{n-2}{r}+\binom{n-2}{r-1}\right)+1.     
   \end{eqnarray*}
From \cite[Example 4.11]{BC2023} we get that $T^1_{-\bfb}(\CU)=0$.
The alternating sum of dimensions in \eqref{eq:les} being equal to zero, we get
\begin{eqnarray*}
  \dim_\K T^2_{-\bfb}(\CU) &=&  \dim_\K H^1\braket{     N_\bfb(\CU)} +
                               \dim_\K H^0\braket{\wt{N}_\bfb(\CU)} -
                               \dim_\K H^0\braket{     N_\bfb(\CU)} \\
                           &=&\textstyle r\binom{n-2}{r}- \binom{n-2}{r-1}.
\end{eqnarray*}
The binomial coefficient $\binom{k}{i}$ is defined as the number of $i$-subsets of a set with $k$ elements.
In particular, it vanishes if $i>k$. Notice also that our proof works for any $r$.
\end{proof}
\begin{corollary}
  \label{cor:unobstructedUnr}
  The uniform matroid $\CU^r_{[n]}$ has vanishing $T^2$ if and only if $n-r\leqslant 2$.
\end{corollary}
\begin{proof}
  From Proposition~\ref{prop:uniformT2} we have that   $T^2_{-\bfb}(\CU_{[n]}^r)=0$ for all $\bfb$ if and only if $n-r\leqslant 2$.
  To conclude,  note that by (\ref{eq:linkUnr}) the difference between the number of elements and the rank stays constant when taking  links.
\end{proof}
\section{Matroids with vanishing $T^2$}
Before we state the main result of this section, we recall a few more notions from matroid theory.
The dual of a matroid $\M$ on $[n]$ is denoted by $\M^\ast$ and is determined by its set of bases:
\[
  \CB_{\M^\ast} = \{ [n]\sdif B \sodass B\in \CB_\M\}.
\]
A proof that $\M$ is a matroid if and only if $\M^\ast$ is a matroid can be found in \cite[Theorem~2.1.1]{Oxl11}.
In particular, if $\rank \M= r$, then  $\rank \M^\ast = n-r$. The later is called the \deff{corank} of $\M$.\\[1ex]
Two elements $v$ and $w$ of a matroid $ \M$ are called parallel if $\{v,w\}\notin\M$.
It is easy to see that parallel elements which are not loops have the same link in \M.
This implies that being parallel defines an equivalence relation on the set of elements of \M\ which are not loops.
A \deff{parallel class} of \M\ is such an equivalence class: a subset $P\subseteq [n]$ containing no loops, such that if $v,w\in P$ with $v\neq w$, then $v\parallel w$. 
The structure of rank two matroids is completely determined by the partition of $[n]\sdif \{\text{loops}\}$ into parallel classes \cite[Proposition~2.2]{CV15}.
Duality implies that isomorphism classes of coloop-free, corank two matroids are in bijection with partitions of $[n]$.\\[1ex]
Recall that a hyperplane of a rank $r$ matroid \M\ is an inclusion-maximal subset $H\subseteq [n]$ such that $\rank\M|_H=r-1$.
In other words, no basis is contained in $H$ and $H$ is maximal with respect to this property.
By \cite[Proposition~2.1.6]{Oxl11} $C$ is a circuit of \M\ precisely when $[n]\sdif C$ is a hyperplane of $\M^\ast$.
This is particularly useful for corank two matroids without coloops because the hyperplanes of $\M^\ast$ are precisely the parallel classes in $\M^\ast$.
So, if $\M$ is a corank two matroid and if $P_1,\dots,P_m$ are the parallel classes  of $\M^\ast$,
then the circuits of $\M$ are the loops and unions of the form
\begin{equation}
    \label{eq:circuitsCorank2}
    P_1\cup \dots\cup\widehat{P_k}\cup\dots P_m.
  \end{equation}
with $k\in\{1,\dots, m\}$.\\[1ex]
Finally, we recall the notion of  \deff{cycle-atomic} subset of $[n]$ that was introduced in [Section~3,BC2023] in order to describe $T^1$ for matroids. What was essentially proved in [Theorem~4.7, BC2023] is that if $T^1_{\bfa-\bfb}(\M)\neq 0$, then $\bfb$ is cycle-atomic. Here we will use the implications of \bfb\ being cycle-atomic to the structure of $N_\bfb$. 
 Here is the definition:
\[
\text{We call }  \bfb \subseteq [n] \text{~cycle-atomic if~}  \bfb\cap C \in \{\emptyset,\bfb\} \text{~for all~} C\in  \CC_\M.
\]
In the next proof we will use some technical results from \cite{BC2023} on $N_\bfb$ and $\wt{N}_\bfb$ when $\bfb$ is cycle-atomic. 
Given the above characterization of circuits in corank two matroids, we note that $\bfb$ is cycle-atomic for a corank two matroid precisely when $\bfb\subseteq P_i$ for some parallel class $P_i$ of $\M^\ast$. \\

%The following statement may follow from more general results on Cohen-Macaulay ideals of codimension 2.
%The next proposition may be approached using the Hilbert-Burch Theorem for codimension two Cohen-Macaulay ideals.
%We believe however that the combinatorial proof we present offers some insight that could serve as a starting point for a proof of Conjecture~\ref{conj:allUnobstructedMatroids}.
\begin{proposition}
  \label{prop:Corank2IsUnobstructed}
  If $\M$ is a matroid of corank at most two, then $T^2(\M)=0$.
\end{proposition}
\begin{proof}
  Notice first that, as all zero-dimensional complexes are matroids, by Lemma \ref{lem:dim0T20} the statement holds for rank one matroids.
  This will allow us to use induction when $\M$ contains a coloop. \\[1ex]
  If the corank is zero or one, then $\K[\M]$ is either a polynomial ring or the coordinate ring of a monomial hypersurface. In both cases $T^2$ vanishes. \\[1ex]
  If the corank is two, then we first notice that $\link_\M A$  has corank at most two for every $A\subset \M$.
  Thus, by (\ref{eq:TiOfLink}), it is enough to prove that $T^2_{-\bfb}(\M)=0$ for every corank two matroid on $[r+2]$ and every $\bfb\subseteq[r+2]$.
  Denote by $P_1,\dots,P_m$ the parallel classes of $\M^\ast$.
  This means, by the above observations, that the circuits of $\M$ are precisely the sets from (\ref{eq:circuitsCorank2}).
  We distinguish two cases.\\[1ex]
  \textbf{Case 1:} $\bfb$ contains no edge of $\M^\ast$.
  This means that there exists an $i\in\{1,\dots,m\}$ such that $\bfb\subseteq P_i$.
  In other words, $\bfb$ is  cycle-atomic.
  So by \cite[Proposition~3.2]{BC2023}, we have that $\wt{N}_\bfb=\emptyset$.
  This implies by (\ref{eq:les}) that $T^2_{-\bfb}(\M)\simeq H^1\braket{N_\bfb}$.
  By \cite[Lemma~4.3]{BC2023} every connected component of $N_\bfb$ contains a unique inclusion-minimal element.
  Thus  each connected component of $\braket{N_\bfb}$ can be contracted to a point, and thus the first cohomology is trivial. \\[1ex]
  \textbf{Case 2:} $\bfb$ contains an edge $e$ of $\M^\ast$.
  The key observation in this case is that the complement of \bfb\ is contained in both $N_\bfb$ and $\wt{N}_\bfb$:
  \[
    [r+2]\sdif\bfb\in N_\bfb \cap \wt{N}_\bfb.
  \]
  On the one hand, $[r+2]\sdif\bfb \subseteq [r+2]\sdif e \in\CB_\M$, thus  it is a face of $\M$ that is disjoint from $\bfb$.
  On the other hand, $([r+2]\sdif\bfb)\cup \bfb = [r+2]\notin\M$, and for every subset $\bfb'\subsetneq\bfb$ of cardinality one less, we  have
  \[
    \#\left(([r+2]\sdif \bfb)\cup \bfb'\right) = r+1.
  \]
  Because $\rank\M =r$, this set  cannot be a face of \M.
  Thus $N_\bfb$ and $\wt{N}_\bfb$ both have the same unique maximal element: $[r+2]\sdif\bfb$.
  This implies that the associated topological spaces are connected and contractible to a point, and we conclude  by (\ref{eq:les}) that $T^2_{-\bfb}(\M)=0$  also in this case.    
\end{proof}
We conjecture that matroids of corank at most two are the building blocks for all  matroids with vanishing $T^2$.
\begin{conjecture}
  \label{conj:allUnobstructedMatroids}
  A matroid $\M$ has $T^2(\M)=0$ if and only if \M\ is the join of matroids of corank at most two.
\end{conjecture}
%We are fairly certain that this conjecture is true.
%By \textit{fairly certain} we mean that we have a plan for its proof.
We provide here an easy argument for the backwards direction, which is a combination of Proposition~\ref{prop:Corank2IsUnobstructed} and the following general statement.\\[1ex]
For any two simplicial complexes $\D$ and $\Gamma$ we have
\[
  \K[\D\ast\Gamma] = \K[\D]\otimes_\K \K[\Gamma].
\]
By a general property of the cotangent modules \cite[Proposition\,2.3]{CI14} we have for any two \K-algebras $A$ and $B$ that
\begin{equation}
  \label{eq:TofJoins}
  T^2(A\otimes_\K B) \simeq \left(T^2(A)\otimes_\K B\right) \oplus \left(A\otimes_\K T^2(B)\right).
\end{equation}
This implies in particular that $T^2(\D\ast\Gamma)=0$, precisely when  $T^2(\D) = T^2(\Gamma) =0$.\\[1ex]
For matroids, the standard terminology is slightly different.
The join operation is often called the direct sum.
In this context, a joint-irreducible matroid is called \emph{connected}  \cite[Section\,4.2]{Oxl11}.\\[1ex]
Besides Proposition~\ref{prop:Corank2IsUnobstructed} and Corollary~\ref{cor:unobstructedUnr}, we have some computational evidence in favor of Conjecture~\ref{conj:allUnobstructedMatroids}. This was obtained using the Macaulay 2 package \verb|versalDeformations| \cite{M2, Ilt12} and the matroid database of simple matroids on at most 9 elements \cite{MR08}.
\bibliographystyle{alpha}

\begin{thebibliography}{ABHL16}

\bibitem[ABHL16]{altmann2016rigidity}
Klaus Altmann, Mina Bigdeli, J{\"u}rgen Herzog, and Dancheng Lu.
\newblock Algebraically rigid simplicial complexes and graphs.
\newblock {\em J. Pure Appl. Algebra}, 220(8):2914--2935, 2016.

\bibitem[AC04]{altmann2000stanleyreisner}
Klaus Altmann and Jan~Arthur Christophersen.
\newblock Cotangent cohomology of {S}tanley-{R}eisner rings.
\newblock {\em Manuscripta Math.}, 115(3):361--378, 2004.

\bibitem[And74]{And74}
Michel Andr{\'e}.
\newblock {\em Homologie des algebres commutatives}, volume 206.
\newblock Springer, 1974.

\bibitem[BC23]{BC2023}
William Brehm and Alexandru Constantinescu.
\newblock The first cotangent cohomology module for matroids.
\newblock {\em Journal of Combinatorial Algebra}, 7(3):327--345, 2023.

\bibitem[Bj{\"{o}}80]{Bjo}
Anders Bj{\"{o}}rner.
\newblock Shellable and {C}ohen-{M}acaulay partially ordered sets.
\newblock {\em Transactions of the American Mathematical Society},
  260(1):159--183, 1980.

\bibitem[CI14]{CI14}
Jan~Arthur Christophersen and Nathan~Owen Ilten.
\newblock Degenerations to unobstructed {F}ano {S}tanley--{R}eisner schemes.
\newblock {\em Mathematische Zeitschrift}, 278:131--148, 2014.

\bibitem[CV15]{CV15}
Alexandru Constantinescu and Matteo Varbaro.
\newblock h-vectors of matroid complexes.
\newblock {\em Combinatorial Methods in Topology and Algebra}, pages 203--227,
  2015.

\bibitem[FN18]{FN18}
Gunnar Fl{\o}ystad and Amin Nematbakhsh.
\newblock Rigid ideals by deforming quadratic letterplace ideals.
\newblock {\em Journal of Algebra}, 510:413--457, 2018.

\bibitem[GS]{M2}
Daniel~R. Grayson and Michael~E. Stillman.
\newblock Macaulay2, a software system for research in algebraic geometry.
\newblock Available at \url{http://www2.macaulay2.com}.

\bibitem[Har09]{HAr09}
Robin Hartshorne.
\newblock {\em Deformation theory}, volume 257.
\newblock Springer Science \& Business Media, 2009.

\bibitem[Hos07]{Hos07}
Richard Hoshino.
\newblock {\em Independence polynomials of circulant graphs}.
\newblock Dalhousie University, 2007.

\bibitem[ICT21]{INT21}
Nathan Ilten, Alfredo~N{\'a}jera Ch{\'a}vez, and Hipolito Treffinger.
\newblock Deformation theory for finite cluster complexes.
\newblock {\em arXiv preprint arXiv:2111.02566}, 2021.

\bibitem[Ilt12]{Ilt12}
Nathan~Owen Ilten.
\newblock Versal deformations and local {H}ilbert schemes.
\newblock {\em Journal of Software for Algebra and Geometry}, 4(1):12--–16,
  2012.

\bibitem[Lod13]{Lod13}
Jean-Louis Loday.
\newblock {\em Cyclic homology}, volume 301.
\newblock Springer Science \& Business Media, 2013.

\bibitem[MR08]{MR08}
Dillon Mayhew and Gordon~F Royle.
\newblock Matroids with nine elements.
\newblock {\em Journal of Combinatorial Theory, Series B}, 98(2):415--431,
  2008.

\bibitem[Nem16]{Nem16}
Amin Nematbakhsh.
\newblock Cotangent cohomology of quadratic monomial ideals.
\newblock {\em arXiv eprint arXiv:1609.06274}, 2016.

\bibitem[Oxl11]{Oxl11}
James Oxley.
\newblock {\em Matroid Theory, Second Edition}.
\newblock Oxford University Press, 2011.

\bibitem[Ser07]{Ser07}
Edoardo Sernesi.
\newblock {\em Deformations of algebraic schemes}, volume 334.
\newblock Springer Science \& Business Media, 2007.

\end{thebibliography}

\end{document}